\documentclass[11pt]{amsart}
\usepackage{amsmath}
\usepackage{amssymb}
\usepackage{multicol}
\usepackage[alphabetic]{amsrefs}
\usepackage{graphicx}
\usepackage{amsfonts, epsfig,color,wrapfig,epstopdf}
\usepackage{xypic,verbatim,amscd,color}
\usepackage{youngtab, ytableau}
\usepackage{graphicx}
\usepackage{colordvi}
\newcommand{\mc}[1]{\mathcal{#1}}
\newcommand{\ov}[1]{\overline{#1}}
\newcommand{\op}[1]{\operatorname{#1}}
\newcommand{\ovmc}[1]{\ov{\mc{#1}}}
\newcommand{\ovop}[1]{\ov{\op{#1}}}

\numberwithin{equation}{section}

\newcommand\tensor{\otimes}

\newcommand{\leto}[1]{\stackrel{#1}{\to}}

\usepackage{amsthm}
\newtheorem{theorem}{Theorem}[section]

\newtheorem{remark}[theorem]{Remark}

\newtheorem{proposition}[theorem]{Proposition}

\newtheorem{lemma}[theorem]{Lemma}

\newtheorem{definition}[theorem]{Definition}
\newtheorem{definition/lemma}[theorem]{Definition/Lemma}
\newtheorem{defi}[theorem]{Definition}
\newcommand{\sL}{\mathfrak{sl}}

\begin{document}

\title[basepoint free divisors on $\ovop{M}_{0,n}$]{basepoint free cycles on $\ovop{M}_{0,n}$ from Gromov-Witten theory}
\author{P. Belkale}
\address{\tiny{Department of Mathematics, University of North Carolina, Chapel Hill \newline \indent Chapel Hill,  NC  27599}}
\email{\tiny{belkale@email.unc.edu} }

\author{A. Gibney}
\address{\tiny{Department of Mathematics, Rutgers University \newline \indent Piscataway, NJ 08854}}
\email{\tiny{angela.gibney@rutgers.edu}}

\maketitle
\begin{abstract}Basepoint free cycles on the moduli space $\ovop{M}_{0,n}$ of stable n-pointed rational curves, defined using Gromov-Witten invariants of smooth projective homogeneous spaces  $X$ are studied.  Intersection formulas to find classes are given, with explicit examples for $X$  a projective space, and  $X$ a smooth projective quadric hypersurface.
When $X$ is projective space,
    divisors are shown equivalent to conformal blocks divisors for type A at level one (studied first in \cite{Fakh}), giving maps from $\ovop{M}_{0,n}$ to birational models constructed as GIT quotients, parametrizing  configurations of weighted points supported on (generalized) Veronese curves (studied first in \cites{KapVer, GiansiracusaSimpson, G}).

\end{abstract}

\section*{introduction and background} Gromov Witten theory has provided important tools to study interesting loci in the moduli space of curves since the early 90's.    Here we use the Gromov-Witten theory of smooth projective homogeneous spaces $X$ to produce a seemingly rich source of basepoint free classes of arbitrary codimension on the moduli space of stable $n$-pointed rational curves.  Basepoint free divisors on a projective variety like $\ovop{M}_{0,n}$ are important as they give rise to morphisms to other projective varieties; basepoint free cycles of higher codimension reflect other aspects of the geometry of the space.

The basic construction of Gromov-Witten classes  is the following: Consider a locus $L$ of points $(C,\vec{p})\in \ovop{M}_{0,n}$, so that there is a stable map  $f$ of some particular degree $\beta$, from a pre-stable curve $(\widetilde{C},\vec{p})$ to a variety $X$, so that $\widetilde{C}$ maps to $C$ and  the images of the marked points $p_i$ lie on some fixed Schubert subvarieties $W_i\subset X$, in general position \cite{KM,FP}\footnote{All definitions, and requirements are explained in Section \ref{oneone}.}. If  $X$ is a homogeneous variety  on which a group $G$ acts transitively, and  the expected dimension of such maps is $-c$, we will find an effective cycle of codimension $c$ on $\ovop{M}_{0,n}$. Moving the $W_i$ by the group $G$, using Kleiman's transversality theorem, one can show that the associated linear system does not have a base locus.

We work with cycles up to rational equivalence, and as we show in  Proposition \ref{GWStrong}, the Gromov Witten loci we consider satisfy a more robust and functorial basepoint free condition, closed under intersection products, which we call  {\em{rational strongly basepoint free}}, after Fulger and Lehmann \cite{FL} (see Definition \ref{SBPFDef}). We often call these strongly base point free.  In Lemma \ref{SBPFprops}, we verify (as with numerical equivalence on smooth varieties as in \cite{FL})  that in this context, the pushforward of strongly base point free classes along flat maps are strongly basepoint free.  Since forgetful maps $\ovop{M}_{0,n}\to \ovop{M}_{0,m}$ with $m<n$ are flat, one obtains base point free divisor classes on $\ovop{M}_{0,n}$ by pushing forward strongly base point free classes, of higher codimension (like the GW classes).  Said otherwise, higher codimension classes are useful even if one is only interested in divisors.

  To identify cycle classes, one may intersect with a dual basis.  For instance, a divisor class on $\ovop{M}_{0,n}$ is computed by intersecting with boundary curves.  Explicit expressions for such an intersection
are given in Propositions \ref{DivisorIntersection}, and \ref{HigherCodimensionIntersection}. Proposition  \ref{Recon} simplifies the formula in Proposition \ref{DivisorIntersection}  considerably in case the  rational cohomology of $X$ is  generated by divisors, as for $X=G/B$ with $B$ a Borel subgroup.

   In practice, to compute the divisor classes, one needs  (1) The (small) quantum cohomology rings of the homogeneous spaces $X$,  and (2) Four point (big) quantum cohomology numbers (where the underlying pointed curve is not held fixed in the enumerative problem). As explained in Proposition \ref{Recon}, it follows from  \cite{KM} that the second condition can be reduced recursively to the first, if the rational cohomology of $X$ is generated by divisors, as for $X=G/B$ with $B$ a Borel subgroup.

In Proposition \ref{complet}, we show that when $X=\Bbb{P}^r$, Gromov Witten divisors are  numerically equivalent to so-called conformal blocks divisors for type A at level one \cite{Fakh} (described in Section \ref{CBDivisors}). In this case classes are indexed by parameters $a_1,\dots,a_n \in \{0,\dots,r\}$, such that $\sum_i a_i\equiv 0 \pmod{r+1}$. Higher codimension GW cycles for $X=\Bbb{P}^r$, especially the basepoint free divisor classes they push forward to under forgetful maps, seem to be new (see Remark \ref{Interesting}).  To illustrate, we give in Section \ref{PushforwardExamples}, two examples of higher codimension cycles that pushforward to divisors.

We also consider examples given by smooth projective quadric
 hypersurfaces  $X=Q_r$ of both even and odd dimension $r$.  In this case, one obtains  a divisor if $\sum_{i}a_i\equiv 1\pmod{r}$.   When $r$ is odd, the  $a_i \in\{0,1,\dots,r\}$ index Schubert classes $H_i$ of codimension $i$ generated by the hyperplane $H_1$.   When even, there are two Schubert classes in the middle dimension.  The cycle classes lie  in cones generated by nef cycles (those classes that nonnegatively intersect ones of complementary codimension).  We see in fact that many divisors formed by quadrics are on extremal faces of the nef cone, as they contract boundary curves depending on the Schubert cycles chosen  (see Propositions \ref{Contract1} and \ref{Contract2}).  We give two examples of extremal rays of the nef cone generated by classes from odd quadrics, and using even quadrics, we give two examples of divisors that lie on extremal faces.  One of the even quadric examples lies
on a two dimensional extremal face of the nef cone not known to be spanned by conformal blocks divisors.

Projective space and quadrics are just the beginning:  In principle, any loci can be studied with available methods.  In Section \ref{Questions} we list two  questions about these and related classes.

\begin{remark}
Loci of enumerative significance inside $(G/B)^n$ were used in  recent work of the first author and J. Kiers \cites{B,BK} to determine the  extremal rays of the $\Bbb{Q}$-cone of $G$-invariant effective divisors on $(G/B)^n$, see \cite[Theorem 1.6]{BK}. These loci bear a resemblance to the Gromov-Witten loci considered in this paper, in that we are  varying the marked curve and keeping the point in $(G/B)^n$ fixed here: In \cite{BK} one considers loci of points $\vec{{g}}\in(G/B)^n$ such that there exist points of $G/P$ which satisfy enumerative constraints given by $\vec{{g}}$. A point in $G/P$ can be viewed as a degree   zero map from a fixed $n$-marked genus zero curve to $G/P$. Maps of non-zero degrees are considered in the multiplicative/quantum generalizations of this problem.

 The Gromov-Witten loci are basepoint free, whereas in \cite{B,BK}, the loci obtained (under some enumerative assumptions) are strongly rigid \cite[Theorem 1.6, (b)]{BK}. It is perhaps fruitful to look at  ``universal"  GW enumerative loci in $\ovop{M}_{0,n}\times (G/B)^n$, but we have not pursued this here.
\end{remark}
\subsection{Methods for obtaining base-point free cycles on $\ovop{M}_{0,n}$} An effective cycle  $\alpha$ of codimension $k$  is basepoint free if the base locus of $\alpha$ is empty:
\begin{definition}\label{bpfDef}A cycle $\alpha\in A_k(Z)$ is basepoint free if for any point $z\in Z$, there is an effective $k$-cycle $\beta$ on $Z$, linearly equivalent to $\alpha$, such that $z$ is not in  the support of $\beta$.
\end{definition}

 A divisor $D$ is basepoint free if and only if the
rational map $\phi_D$ defined by $D$ is a morphism.  Naturally as basepoint free divisors on $\ovop{M}_{0,n}$
correspond to morphisms from $\ovop{M}_{0,n}$ to projective varieties, there has been interest in determining such divisors.

 A number of morphisms from the moduli space of curves to other projective varieties have been found and studied in the literature, many giving alternative compactifications. These new moduli spaces  include for example, cyclic covers \cite{FedCyclic}, and GIT quotients \cites{GiansiracusaSimpson, AS, G, GG, GJM, gjms} generalizing Kapranov's compactifications of $\op{M}_{0,n}$ \cites{KapVer,KapChow}.  The latter  giving a different (but overlapping) set of modular interpretations than described in the work of Smyth \cite{Smyth}, in that they parametrize embedded curves, as opposed to abstract ones (cf. \cite[p.245]{GJM}).  Other morphism have been found as well.

The following other methods are known for obtaining basepoint free cycle classes on $\ovop{M}_{0,n}$:
\begin{enumerate}
\item[(1)] First Chern classes of conformal block bundles (computed in genus zero by \cite{Fakh}: These are associated to irreducible representations $\lambda_1,\dots,\lambda_n$ at a level $\ell$ of a simple Lie group $G$). Also Schur polynomials  in the Chern classes of conformal block bundles (see Remark \ref{CBPush} and Lemma \ref{SBPFprops}; also see \cite[Example 12.1.17]{Fulton}).
\item[(2)] Gromov-Witten classes $I^{c,X}_{\beta,\vec{\alpha}}$ with $X$ homogeneous (see Prop \ref{GWStrong}).
\item[(3)] Algebraic operations in (1) and (2): intersection products, pushforwards under point-dropping maps $\ovop{M}_{0,n}\to \ovop{M}_{0,m}$, with $m<n$, and iterations of these (see Prop \ref{GWStrong} and Lemma \ref{SBPFprops}).
\end{enumerate}

The only verified identity above are for the GW divisors with $c=1$ for $X=\Bbb{P}^r$ and the level one conformal block divisors for $A_r$ (Proposition \ref{complet}).   We have seen possible
connections, at least on the level of parameters, between GW divisors for Grassmannian varieties and for higher
level type A conformal blocks.
\begin{remark}
GW divisor classes for $\mathbb{P}^{r}=Gr(1,r+1)$, a homogenous space for $\sL_{r+1}$ for $c=1$, coincide with conformal blocks divisors for $\sL_{r+1}$ at level one, and there may be a more general connection between GW divisors for Grassmannians $Gr(\ell,r+1)$ and  conformal blocks divisors for $\sL_{r+1}$ at higher levels $\ell$ . The known level one identity is different from the type of pairing in Witten's theorem where small quantum numbers for $\mathbb{P}^{r}=\op{Gr}(1,r+1) =\op{Gr}(r,r+1)$ are paired with ranks of CB divisors for $\sL_{r}$ (and not $\sL_{r+1}$). On the other hand, basepoint free divisors produced by conformal blocks, and by GW theory, are generally speaking  parameterized by different data: Conformal blocks by representations of a semisimple group at a level, and GW cycles by  Weyl group data for a homogeneous space.   We would be surprised if  the two were always equal.
\end{remark}
 \subsection{Gromov-Witten theory preliminaries}\label{oneone}

For nonnegative integers $(g,n)$ such that $3g-3+n\ge 0$,  there is an irreducible, projective variety  $\ovop{M}_{g,n}$
whose points are in one-to-one correspondence with isomorphism classes of curves with at worst nodal singularities and only finitely many automorphisms.   The stack $\overline{\mathcal{M}}_{g,n}$,  reflects the geometry of $\overline{\operatorname{M}}_{g,n}$, while being in certain ways easier to study. Since $\ovop{M}_{0,n}$ is a fine moduli space, these two points of view are equivalent in this case.   For most of the paper, we consider the case when $g=0$, where the stack is represented by the smooth projective variety $\overline{\op{M}}_{0,n}$. For this reason, we routinely refer to $\ovmc{M}_{0,n}$ as a space instead of stack.

The set of all stable maps of genus $g$ and degree  $\beta \in H_2(X)$, with $n$ marked points to a normal variety $X$  forms a (Deligne-Mumford) moduli stack $\ovmc{M}_{g,n}(X,\beta)$.  Stable maps are tuples
$((C,\vec{p}), f)$, where $(C,\vec{p})$ is a pre-stable curve and $f$ is a stable map from $(C,\vec{p})$ to $X$.
A pre-stable curve $(C,\vec{p})$ is a connected, and reduced curve $C$ of genus $g$, with at worst  nodal singularities, and $\vec{p}$ are smooth points on $C$.  A stable map is any morphism from a pre-stable curve to $X$ such that there are only finitely many automorphisms of the map.

 To construct Gromov-Witten invariants,  one uses the $n$ evaluation maps
 $\op{ev}_i: \ovmc{M}_{g,n}(X,\beta)\to X$,  the contraction map
 $\eta:\ovmc{M}_{g,n}(X,\beta)\to  \ovmc{M}_{g,n}$.
  The virtual fundamental class $[\ovmc{M}_{g,n}(X,\beta)]\in A_{\nu,\Bbb{Q}}(\ovmc{M}_{g,n}(X,\beta))$ is  constructed in  \cite{BF}.
The virtual dimension is
\begin{equation}\label{expected}
\nu=(3g-3+n) + c_1(T_X)\cdot\beta +  (1-g)\dim X.
\end{equation}
Given $\alpha_1,\dots,\alpha_n\in A_{\Bbb{Q}}^*(X)$, $\alpha_i\in A^{|\alpha_i|}(X)$, the class of $W_i$, the Gromov-Witten class
$$I^X_{g,n,\beta}(\alpha_1\tensor\alpha_2\tensor\dots\tensor\alpha_n)\in A_{e,\Bbb{Q}}( \ovmc{M}_{g,n})$$
is obtained as the  push forward by the contraction $\eta:\ovmc{M}_{g,n}(X,\beta)\to  \ovmc{M}_{g,n}$, of the
cap product
\begin{equation}\label{class}
\prod_{i=1}^n\op{ev}_i^*(\alpha_i)\cap [\ovmc{M}_{g,n}(X,\beta)]\in A_{e,\Bbb{Q}}(\ovmc{M}_{g,n}(X,\beta))=A^c_{\Bbb{Q}}(\ovmc{M}_{g,n}(X,\beta))
\end{equation}
where
$$e=\nu-\sum|\alpha_i|=(3g-3+n) + c_1(T_X)\cdot\beta + (1-g)\dim X -\sum|\alpha_i|.$$
It is a cycle  of codimension
\begin{equation}\label{codim}
c=\dim \ovmc{M}_{g,n}-e=\sum|\alpha_i|-c_1(T_X) \cdot \beta- (1-g)\dim X.
\end{equation}

\begin{definition}\label{DivisorCondition}
We say that a triple  $(X,\beta, \vec{\alpha})$ satisfies the {\bf{co-dimension c cycle condition}} if
\begin{equation}\label{BPFClass}
\sum_i |\alpha_i|=c+c_1(T_X) \cdot \beta+ (1-g)\dim X.
\end{equation}
\end{definition}

\begin{defi}
For arbitrary homogeneous  $\alpha_1,\dots,\alpha_n$, define the GW-cycles $I^{c,X}_{\beta,\vec{\alpha}}\in A^c_{\Bbb{Q}}(\ovmc{M}_{g,n})$, on $\ovmc{M}_{g,n}$ as follows:
$$I^{0,X}_{\beta,\vec{\alpha}}= \left\{
\begin{matrix}
d & \text{if} \  (X,\beta,\vec{\alpha}) \  \text{satisfies the codimension $c=0$ cycle condition}, where \\
&  I^X_{0,n,\beta}(\alpha_1\tensor\alpha_2\dots\tensor\alpha_n)=d [\ovmc{M}_{g,n}]\in A^0_{\Bbb{Q}}(\ovmc{M}_{g,n});\\
0 &  \text{otherwise}.
\end{matrix}
\right.$$

$$I^{c,X}_{\beta,\vec{\alpha}}= \left\{
\begin{matrix}
I^X_{0,n,\beta}(\alpha_1\tensor\alpha_2\tensor \dots\tensor\alpha_n) & \text{if} \  (X,\beta,\vec{\alpha}) \  \text{satisfies the codimension $c>0$ cycle condition}; \\
0 & \text{otherwise}.
\end{matrix}
\right.$$
\end{defi}
\begin{remark}
Localization techniques are used to compute these invariants in many cases, especially for homogeneous $X$ \cite{GP}.
\end{remark}

\section{Rational strongly base point freeness and GW-cycles}\label{FuLe}
Here we define the notion of rational strongly basepoint free cycles, which is inspired by the one given in \cite{FL} for strongly basepoint free cycles. In Lemma \ref{SBPFprops} we list a number of properties satisfied by such strongly basepoint free cycles.
In Proposition \ref{GWStrong}, we show that Gromov-Witten cycles $I^{c,X}_{\beta,\vec{\alpha}}$ with $X$ homogeneous, are rationally strongly base point free.
In Remark \ref{CBPush}, we point out that Schur polynomials in the Chern classes of $\mathbb{V}(\mathfrak{g},\vec{\lambda},\ell)$ are strongly base point free on $\ovop{M}_{0,n}$.

Recall that forgetful maps $\ovop{M}_{0,n}\to \ovop{M}_{0,m}$ with $m<n$ are flat. Lemma \ref{SBPFprops} together with
Propositions \ref{GWStrong} and \ref{CBPush} therefore are a source of basepoint free cycles on the moduli spaces $\ovop{M}_{0,n}$. In particular, one obtains base point free classes on $\ovop{M}_{0,n}$ by pushing forward strongly base point free classes of higher codimension on suitable $\ovop{M}_{0,n'}$ with $n'>n$.
\subsection{Rational strongly basepoint free cycles}
\begin{defi}\label{SBPFDef} An effective integral Chow cycle $\alpha\in A^k(X)$ of  codimension $k$ on an equidimensional possibly singular,  reducible, and/or disconnected projective variety $X$ is said to be {\em{rationally strongly basepoint  free}}  if there  is a  flat  morphism $s:U\to X$ from  an  equidimensional  quasi-projective  scheme $U$ of finite type   and a  proper  morphism $p:U\to W$ of  relative  dimension $\dim X-k$,  where $W$ is  an irreducible quasi-projective variety,  isomorphic to an open subset of $\Bbb{A}^m$ for a suitable $m$,  such that  each  component  of $U$ surjects  onto $W$, and  $\alpha=  (s|F_p)_*[F_p]$,  where $F_p$ is  a  general  fiber  of $p$.
\end{defi}
\begin{defi}
Denote the semigroup of rationally strongly basepoint free classes of codimension $k$ on a (possibly singular) projective variety $X$ by $\op{SBPF}^k(X)\subseteq A^k(X)$.\end{defi}


For rationally strongly basepoint free cycles, unlike for the strongly basepoint free cycles of \cite{FL}, we are working with rational equivalence, rather than numerical equivalence.  Moreover, for $\op{SBPF}^k(X)\subseteq A^k(X)$, we do not form the closure of the cones generated by such classes. We have included the condition that $W$ is an open subset of $\Bbb{A}^m$ since we are interested in rational equivalence.
Moreover, one can drop the condition that each component of $U$ surjects onto $W$ since $W$ is required to be quasi-projective, we may replace it by an open subset, and $U$ by the inverse image of this open set.

 Note  that if $F_{p_i}$, $i=1,2$ are two fibers then  $(s|F_{p_i})_*[F_{p_i}]$ coincide in $A^k(X)$.  Indeed, suppose
 $U$ sits inside a projective space $\Bbb{P}\times W$ over $W$, and  $\overline{W}$ is a projective space containing $W$ as an open subset. Form the closure $\overline{U}$ of $U$ in the projective variety $\Bbb{P}\times \overline{W}\times X$. We have maps
$\overline{U}\to X$ (which may not be flat) and $\overline{U}\to \overline{W}$. Over $W\subseteq\overline{W}$, $U$ and $\overline{U}$ coincide. Therefore $F_{p_i}$ are also fibers of  $\overline{U}\to \overline{W}$ and are hence rationally equivalent. Now $\overline{U}\to X$  is proper and hence the pushforwards of the fibers agree in Chow groups.

\begin{lemma}\label{SBPFprops}Rationally strongly basepoint free classes satisfy the following properties:
\begin{enumerate}
\item[(a)] A rationally strongly basepoint free class $\alpha\in \op{SBPF}^k(Z)$ is basepoint free in the following stronger sense: Given any irreducible subvariety $V\subset Z$ (for example a point), there is a effective cycle of class $\alpha$ which intersects $V$ in no more than expected dimension (if the intersection is non-empty).
\item[(b)] If $Z$ is a smooth projective variety and $\alpha\in \op{SBPF}^k(Z)$ and $\beta\in \op{SBPF}^{k'}(Z)$, then their intersection product $\alpha\cdot\beta\in \op{SBPF}^{k+k'}(Z)$.
\item[(c)] Let $\pi:X\to Y$ be a flat morphism of relative dimension $d$ and $\alpha\in \op{SBPF}^k(X)$, then $\pi_*\alpha\in \op{SBPF}^{k-d}(Y)$.
\item[(d)] If $X,Y$ are projective varieties, with $Y$ smooth, and $\pi:X\to Y$ is a morphism, then $\pi^*\op{SBPF}^k(Y)\subseteq \op{SBPF}^k(X)$.
\item[(e)] The cycle class of a Schubert variety on a $G/P$ is rationally strongly base point free. Therefore all effective cycles on a homogeneous space are rationally strongly base point free.
\item[(f)]  Let $\Bbb{V}$ be a globally generated vector bundle of rank $n$ on a smooth projective variety $X$. The Schur polynomial $s_{\lambda}=\det(c_{{\lambda_i}+j-i})_{1\leq i,j\leq n}$ in the Chern classes $c_i=c_i(\Bbb{V})$ of  $\Bbb{V}$ lies in $\op{SBPF}^{|\lambda|}(Z)$. Here $|\lambda|=\sum |\lambda_i|$ is the length of the partition $\lambda=(\lambda_1\geq \dots\geq \lambda_n\geq 0)$.  See \cite[Def 3.2]{FL} and the proof of \cite[Lemma 5.7]{FL}.
\item[(g)] Base point free divisors on a smooth variety are rationally strongly base point free.
\end{enumerate}
\end{lemma}
\begin{proof}For (a), $\dim (V\cap s(F_p))\leq \dim (s^{-1}(V)\cap F_p)$, which in turn is the generic dimension of fibers of $s^{-1}(V)\to W$ which is $\dim U -\dim X +\dim V-\dim W=\dim V-k$. Part (b) follows from \cite[Corollary 5.6]{FL}. The $W$ for the intersection cycle is the product of the $W$ for $\alpha$ and $\beta$ and is hence rational. Part (c) follows from the same  proof as \cite[Lemma 5.3]{FL} (here the $W$ is unchanged for the pushforward): In particular, it isn't necessary
 to assume that $X$ or $Y$ are smooth. For (d) see  \cite[Lemma 5.4]{FL} (particularly the first paragraph of the proof there, the $W$ is unchanged here). In particular, one does not need smoothness of $X$. Property (e) follows by taking $W=G/B$ (which is rational), $U$ the universal Schubert variety in $G/B\times G/P$, and $X=G/P$.
 Statement $(f)$ was proved for strongly basepoint free cycles (see \cite[Def 3.2]{FL} and the proof of \cite[Lemma 5.7]{FL})  for smooth varieties, is true for singular projective varieties $X$ as well (using properties (d) and (e) with $Y$ a Grassmannian).
 For property (g), note that any base point free divisor on a smooth variety the  pull back, from a projective space $\Bbb{P}^n$, of an effective divisor by a morphism and hence properties (c) and (e) apply.
 \end{proof}

\subsection{GW classes are rationally strongly basepoint free}\label{GWStrongSection}
For the rest of the paper we assume $g=0$,  the variety
$X=G/P$, is homogeneous, and  $\alpha_i$ cycle classes of Schubert varieties.  By \cite{FP}, the coarse moduli space $\ovop{M}_{0,n}(X,\beta)$ is equidimensional of the expected dimension \eqref{expected} (with $g=0$), and we may work with the fundamental class of the coarse moduli space $\ovop{M}_{0,n}(X,\beta)$  instead of the virtual fundamental class. The classes $I^{c,X}_{\beta, \vec{\alpha}}$ are therefore integral Chow cycles.

\begin{proposition}\label{GWStrong}
Assume $X=G/P$, and let $(X, \beta, \vec{\alpha})$ satisfy the co-dimension $c$ cycle condition. Then the Gromov Witten  cycle  $I^{c,X}_{\beta, \vec{\alpha}}$ is  a rationally strongly basepoint free on $\ovop{M}_{0,n}$, i.e.,  $I^{c,X}_{\beta, \vec{\alpha}}\in \op{SBPF}^k(\ovop{M}_{0,n})$.
\end{proposition}

To prove Proposition \ref{GWStrong}, we refer to the following.

\begin{lemma}\label{below}Let $\eta:\ovop{M}_{0,n}(X,\beta)\to \ovop{M}_{0,n}$, and $x \in  \ovop{M}_{0,n}$.
\begin{enumerate}
\item Each component of $\eta^{-1}(x)$ has dimension equal to $\dim \ovop{M}_{0,n}(X,\beta)-\dim \ovop{M}_{0,n}$;
\item The map $\eta$ is flat.
\end{enumerate}
\end{lemma}
\begin{proof}(of Lemma \ref{below})
{\em{Part} (1):} This is of course well known, and follows from \cite{KM} and \cite{FP}.
For a fixed nodal curve $C$ of arithmetic genus $0$, the space of maps $C\to X$ has dimension $\dim X + \beta\cdot T_X= \dim \ovop{M}_{0,n}(X,\beta)-\dim \ovop{M}_{0,n}$ \cite[Section 5.2]{FP}. We have to therefore account for the collapsing operation
in which a $C$ has a component which is mapped on to $X$ with positive degree, and has only two special points (the point in $\ovop{M}_{0,n}$ collapses this component). Such maps are subject to a non-trivial equivalence: The extra component
has positive dimensional space of automorphisms fixing the marked points, and hence brings down the count of space of maps by at least one.

{\em{Part} (2):} Locally $\ovop{M}_{0,n}(X,\beta)$ is the quotient of a smooth variety $Y$ by a finite group $G$. The composite map $Y\to Y/G\subseteq\ovop{M}_{0,n}(X,\beta)\leto{\eta}\ovop{M}_{0,n}$ is flat since $Y$ is smooth and all fibers have the expected dimension by Lemma \ref{below}. Now the coordinate rings of $Y/G$ are direct summands of the coordinate rings of $Y$, and hence are flat over $\ovop{M}_{0,n}(X,\beta)$ (see \cite[Remark 2.6.8]{KV}).
\end{proof}

\begin{proof} (of Proposition  \ref{GWStrong}) One has maps $\op{ev}: \ovop{M}_{0,n}(X,\beta)\to X^n$ and flat maps $\eta:  \ovop{M}_{0,n}(X,\beta)\to \ovop{M}_{0,n}$. We claim that the pull back under $\op{ev}$ of $\alpha_1\tensor\alpha_2\tensor\dots\tensor\alpha_n$ is strongly base point free. This claim implies Proposition \ref{GWStrong}, since the Gromov-Witten cycle  $I^{c,X}_{\beta, \vec{\alpha}}$ is the pushforward $\eta_*(\op{ev}^* ( \alpha_1\tensor\alpha_2\tensor\dots\tensor\alpha_n))$ and $\eta$ is flat (and using Property (c) in Section \ref{FuLe}).

Every effective cycle  on a projective homogeneous space is strongly basepoint free, see  Lemma \ref{SBPFprops} (e). Therefore $\alpha_1\tensor\alpha_2\tensor\dots\tensor\alpha_n$ is a strongly basepoint free cycle on $\ovop{M}_{0,n}(X,\beta)$.

We now use the following property (d) of Lemma \ref{SBPFprops}: If $X,Y$ are projective varieties, with $Y$ smooth, and $\pi:X\to Y$ is a morphism, then $\pi^*\op{SBPF}^k(Y)\subseteq \op{SBPF}^k(X)$.
\end{proof}
\begin{remark}
It is easy to see, rather immediately, that $I^{c,X}_{\beta, \vec{\alpha}}$ is  basepoint free on $\ovop{M}_{0,n}$. Let $P$ be a point of $\ovop{M}_{0,n}$, and let $Z\subset X^n$ be the product of Schubert  varieties $X_i$ with cycle classes $\alpha_i$. Note that $G^n$ acts transitively on $X$. By Kleiman's Bertini theorem \cite{Kleiman}, for general $\vec{h}=(h_1,\dots,h_n)\in G^n$, one has that $\op{ev}^{-1}(\vec{h}Z)$ has the expected codimension inside $\ovop{M}_{0,n}(X,\beta)$, and meets the fiber $\eta^{-1}(P)$ (which is equidimensional) in the expected dimension, which is easily computed to be $-c<0$. The cap product \eqref{class} is represented by the effective cycle $\op{ev}^{-1}(\vec{h}Z)$, and the basepoint freeness follows.
\end{remark}
\subsection{Chern classes of conformal blocks on $\ovop{M}_{0,n}$ are rationally strongly basepoint free}\label{CB}\label{CBDivisors}

Conformal blocks bundles refers to first vector bundles of coinvariants $\mathbb{V}(\mathfrak{g},\vec{\lambda},\ell)$ defined on $\ovmc{M}_{g,n}$, where $(\mathfrak{g},\vec{\lambda},\ell)$ is a compatible triple consisting of a simple Lie algebra $\mathfrak{g}$,  a positive integer $\ell$, and $\vec{\lambda}=(\lambda_1,\ldots,\lambda_n)$ are dominant weights for $\mathfrak{g}$ at level $\ell$. One can find a construction of these bundles in \cite{Fakh} (they were originally constructed in \cite{TUY}), as well as a proof of global generation in case $g=0$, and many relevant examples and results, including formulas for  the  Chern classes in genus zero. Formulas for the first Chern classes were given in the cases of genus zero, and genus one with one marked point in \cite{Fakh}. Together with factorization formulas, these determine the first Chern class in any genus (such F-curves span the second homology).  Formulas for the total Chern character were given in \cite{MOPPZ} in arbitrary genus; and are referred to as Verlinde bundles there.

\begin{remark}\label{CBPush}
Schur polynomials in the Chern classes of $\mathbb{V}(\mathfrak{g},\vec{\lambda},\ell)$ on $\ovop{M}_{0,n}$ are strongly base point free. Note that these Schur classes  include Chern classes of $\Bbb{V}$.
Indeed, the vector bundles $\mathbb{V}(\mathfrak{g},\vec{\lambda},\ell)$ defined on $\ovop{M}_{0,n}$, are globally generated in case $g=0$, and parts (e),(f) of Lemma \ref{SBPFprops} therefore apply.
\end{remark}

\section{GW cycles}

In Proposition \ref{DivisorIntersection} we give a formula for the intersection of a GW cycle of codimension one with F-Curves, described below in Def \ref{FCurve}.  These curves can be used to compute the class of a divisor (see Section \ref{NA}).
 Ingredients for the proof of Proposition \ref{DivisorIntersection} will be defined in Section \ref{FactProp}.  The proof is given  in Section \ref{DivIntProof}.

Prop \ref{DivisorIntersection} is generalized in Proposition \ref{HigherCodimensionIntersection} to
give an explicit formula for the intersection of a GW-loci $I^{c,X}_{\beta,\vec{\alpha}}$  of arbitrary codimension $c$ with a boundary cycle of complementary codimension, which like $\op{F}$-curves, are products of moduli spaces.    The proof of Prop \ref{DivisorIntersection} that of Proposition \ref{DivisorIntersection}, and we state them separately for clarity, and because we focus on divisors.

We show in Section \ref{Reconstruction} how it is sometimes possible to simplify the formulas by reducing four-point classes
 to three points.

\subsection{Computing classes of GW cycles by intersecting with boundary classes}\label{divisorClasses}

\subsubsection{Intersecting GW divisors with boundary curves}
\begin{definition}\label{FCurve}
  If $N_1\cup \cdots \cup N_4$ is a partition of $[n]=\{1,\ldots,n\}$ consisting of four nonempty subsets, then given four points
   ($\mathbb{P}^1, \{p_i\}_{i \in N_j}\cup P_j)\in \ov{M}_{0,|N_j|+1}$, we can define a map
  $$\ovop{M}_{0,4} \longrightarrow \ovop{M}_{0,n}, \ \ (C_0, \{Q_1,\ldots,Q_4\}) \mapsto (C,\vec{p}),$$
 where $C$ is a union of $C_0$ and the four copies of pointed $\mathbb{P}^1$ glued
by attaching  the points $\{P_j\}_{j=1}^4$ to the four marked points $\{Q_j\}_{j=1}^4$.   The $\op{F}$-Curve $F_{N_1,\cdots,N_4}$ is the numerical equivalence class of the image of this map.
 \end{definition}

\begin{proposition}\label{DivisorIntersection}Let $\op{F}_{N_1,\ldots,N_4}$ be an $\op{F}$-Curve on $\ovop{M}_{0,n}$, let $X$ be a smooth projective homogeneous variety and suppose $\vec{\alpha}$ satisfies the codimension one co-cycle condition.  Then
$$I^{1,X}_{\beta,\vec{\alpha}} \cdot \op{F}_{N_1,\ldots,N_4} = \sum_{\stackrel{\vec{\omega}=\{\omega_1,\ldots,\omega_4\}}{ \in (W/W_P)^4}}
 I^{1,X}_{\beta-\sum_{j=1}^4\beta_j, \vec{\omega}}  \prod_{j=1}^4 \ I^{0,X}_{\beta_j, \alpha(N_j)\cup \omega_j'} .$$
\end{proposition}

We note the similarity of the expression in the statement Prop \ref{DivisorIntersection} with \cite[Prop 2.7]{Fakh} which gives the intersection of conformal blocks divisors $c_1(\mathbb{V}(\mathfrak{g},\vec{\lambda},\ell))$ and F-Curves.  These are equal in the case $X=\mathbb{P}^r$,  $\mathfrak{g}=\sL_{r+1}$, and $\ell=1$ (see Prop \ref{complet}).

\subsubsection{The nonadjacent basis}\label{NA}
To compute classes of GW divisors in examples, we will use what is called the nonadjacent basis, which we next describe.  Let $G_n$ be a cyclic graph with $n$ vertices labeled
$S=\{1, 2, \ldots , n\}$.  A subset of vertices $T\subset S$ is called adjacent if  $t(T)$, the number of connected components of the subgraph generated T, is 1. Since $G_n$ is cyclic, if $t(T) = k$, then $t(T^c) = k$.

By \cite[Proposition 1.7]{Carr}. The set $B=\{\delta_{T}: \ t(T) \ge 2\}$  forms a basis of $\op{Pic}(\ovop{M}_{0,n})_{\mathbb{Q}}$.

The dual of a basis element $\delta_{T} \in B$ is an $\op{F}$-curve if and only if $t(T)=2$, and for $t(T)>2$, dual elements are alternating sums of $\op{F}$-curves.  In \cite{MoonSwin} an algorithm is given for finding a dual element.  For the purposes of computing examples we give the dual basis for $n=5$ and $n=6$ below.

On $\ovop{M}_{0,5}$, one nonadjacent basis is given by  $B=\{\delta_{13}, \delta_{14}, \delta_{24}, \delta_{25}, \delta_{35}\}$, and the dual basis to $B$ consists of $\op{F}$-curves
$\{\op{F}_{1,2,3,45}, \op{F}_{1,4,5,23}, \op{F}_{2,3,4,15}, \op{F}_{1,2,5,34}, \op{F}_{3,4,5,12}\}$.
For $\op{Pic}(\ovop{M}_{0,6})$ one nonadjacent basis is $$\{\delta_{13},\delta_{14},\delta_{15},\delta_{24},\delta_{25},\delta_{26},\delta_{35},\delta_{36},\delta_{46},
\delta_{124},\delta_{125},\delta_{134},\delta_{135},\delta_{136},\delta_{145},\delta_{146}\}.$$ Classes of divisors can be computed by intersecting with curves in the dual basis:
\begin{multline}
\{F_{1,2,3,456},F_{1,4 ,23 ,56},F_{1, 5,6 ,234}, F_{2, 3, 4, 156},  F_{2, 5, 16, 34},F_{1, 2, 6, 345},\\
F_{3, 4, 5, 126}, F_{3, 6, 12, 45}, F_{4, 5, 6, 123},
F_{3, 4, 12, 56}, F_{5, 6, 12, 34}, F_{1, 2, 34, 56}, \\ (F_{5, 6, 13, 24} + F_{1, 2, 3, 456}+F_{2, 3, 4, 156} - F_{2, 3, 16, 45}), F_{2, 3, 16, 45},
F_{1, 6, 23, 45}, F_{4, 5, 16, 23}\}.
\end{multline}

\subsubsection{Intersecting higher codimension GW cycles with boundary classes} For $k=n-3-c$, the locus
$$\delta^k(\ovop{M}_{0,n})=\{(C,\vec{p}) \in \ovop{M}_{0,n} | \ \mbox{ C has at least k nodes } \}$$
is effective and has dimension $c$.  We will next give a formula for the intersection of its irreducible components with $I^{c,X}_{\beta,\vec{\alpha}}$ in case $(X,\beta, \vec{\alpha})$ satisfies the codimension c cycle condition.  For the formula, we set a small amount of notation.
Irreducible components of $\delta^k(\ovop{M}_{0,n})$ are determined by the dual graph of the curves parametrized.  Such a graph is a tree with $k$ edges, joining $k+1$ vertices, decorated by $n$ half-edges, so that each vertex has at least 3 edges plus half-edges.

To simplify the discussion, we label the vertices $\vec{v}=\{v_1,\ldots,v_{k+1}\}$, and edges $\vec{e}=\{e_{ij}\}_{1\le i<j\le k+1}$, where we take $e_{ij}$ to be zero unless $v_i$ and $v_j$ are connected by an edge. Half-edges are labeled $\vec{h}=\{h_j\}_{j=1}^n$, and we label the component $\delta^k(\Gamma_{\vec{v},\vec{e},\vec{h}})$.

In the formula given in Proposition \ref{HigherCodimensionIntersection}, given a vertex $v_i$, by $\alpha(v_i)$ we mean the set of $\alpha_j \in A^*(X)$ associated to the set of half edges $h_j$ attached to the vertex $v_i$.  For each vertex $v_i$ we'll also consider new classes $\gamma_{ia} \in A^*(X)$, associated to the nonzero edges $e_{ia}$ for $i+1 <a < k+1$ and classes  $\gamma_{ai}^* \in A^*(X)$, dual to $\gamma_{ai} \in A^*(X)$, associated to each nonzero edge $e_{ai}$.  If the edge $e_{ij}$ is zero (so vertices $v_i$ and $v_j$ are not connected in the dual graph), we still write down a class  $\gamma_{ij}$, but it is simply not in the formula, or one can imagine that there is an edge, and, by propagation of vacua, we may assume the class is zero.

\begin{proposition}\label{HigherCodimensionIntersection}With notation as above, one has
$$I^{c,X}_{\beta,\vec{\alpha}} \cdot \delta^k(\Gamma_{\vec{v},\vec{e},\vec{h}})
= \prod_{\overset{1\le i \le k+1}{ 0\le \beta_i \le \beta, \sum_{i=1}^{k+1}\beta_i = \beta}}
\prod_{\overset{\gamma_{1i},\ldots,\gamma_{i-1i},\gamma_{i i+1},\ldots, \gamma_{i k+1}}{
\in (W/W_P)^{k+1}}}
I^{c_i,X}_{\beta_i, \alpha(v_i)\cup \{\gamma^*_{ai}\}_{a=1}^{i-1} \cup \{\gamma_{ia}\}_{a=i+1}^{k+1}}.$$
\end{proposition}

While we mainly focus on examples of divisors here,
there are a number of reasons to compute classes of higher codimension GW cycles.
For instance, while one usually obtains base point free classes by pullback along a morphism, one of the main themes of this paper is that the pushforward of higher codimension cycles produces new basepoint free divisor classes.
Moreover, it is not clear even in the simplest case  for $X=\Bbb{P}^r$, what divisors we get by push forwards from higher codimension ($c>1$). These classes for $r=2$ will incorporate the Kontsevich counts of number of rational curves in $\Bbb{P}^2$ passing through  a fixed number of points in general position: To the best of our knowledge, these Kontsevich counts are not known to be related to representation theoretic or conformal blocks ranks, and therefore any relation to Chern classes of conformal blocks is new to us.

\subsection{Ingredients for the proofs of Propositions  \ref{DivisorIntersection} and \ref{HigherCodimensionIntersection}}
\subsubsection{Factorization, Propagation of Vacua}\label{FactProp}
In the proof of Propositions  \ref{DivisorIntersection} and \ref{HigherCodimensionIntersection} we use two properties of GW classes which we call Factorization, and Propagation of Vacua for their similarity to
properties of the same name that hold for vector bundles of coinvariants and conformal blocks.  Here $I^{0,X}_{\beta, \vec{\alpha}}$  plays the role of the rank of the vector bundle of co-invariants,
  and the $I^{1,X}_{\beta, \vec{\alpha}}$ corresponds to first Chern classes of the bundles.

To state this {factorization formula} \cite[Section 2.2.6]{KM}, we write the cohomology class of the diagonal for $X=G/P$: The Schubert classes $X_w$ in $X=G/P$ are parameterized by $W/W_P$.  For $w\in W/W_P$, let $w'$ be the unique element so that $[X_w]\cdot [X_{w'}]=[pt]\in H^*(X)$.  Then the cohomology class of the diagonal $\Delta\subset X\times X$ is
\begin{equation}\label{diagonal}
[\Delta]=\sum_{w\in W/W_P} X_w\tensor X_{w'}\in A^{\dim X}(X\times X).
\end{equation}

Let $\gamma: \ovop{M}_{0,n_1+1} \times \ovop{M}_{0,n_2+1} \to \ovop{M}_{0,n_1+n_2}$
be the clutching morphism, where one attaches pointed curves by glueing them together
along the last marked point for each factor. Let $\pi_i:\ovop{M}_{0,n_1+1} \times \ovop{M}_{0,n_2+1} \longrightarrow  \ovop{M}_{0,n_i+1}$ be the projection maps.

If $(X,\beta, \{\alpha_1,\ldots,\alpha_{n_1+n_2}\})$ satisfies the codimension $1$ cycle condition, then the  {\bf{factorization formula}} states that
$\gamma^* I^{1,X}_{\beta, \{\alpha_1,\dots,\alpha_{n_1+n_2}\}}$ decomposes as sum of divisor classes pulled back from the $\ovop{M}_{0,n_i+1}$ along $\pi_i$. The class pulled back from
$\ovop{M}_{0,n_2+1}$ equals
$$\sum_{\beta_1 +\beta_2=\beta, w\in W/W_P}I^{0,X}_{\beta_1, \{\alpha_1,\dots,\alpha_{n_1},[X_{w}]\}}\pi_2^* I^{1,X}_{\beta_2, \{\alpha_{n_1+1},\dots,\alpha_{n},[X_{w'}]\}}$$

Note that if $c_{\beta_1}$ and $c_{\beta_2}$ are the corresponding co-dimensions in \eqref{codim} then
$c_{\beta_1}+c_{\beta_2} =c _{\beta}$, since the co-dimensions of $X_w$ and $X_{w'}$ add up to $\dim X$.

If $(X,\beta, \{\alpha_1,\ldots,\alpha_{n_1+n_2}\})$ satisfies the codimension $0$ cycle condition, then  $I^{0,X}_{\beta, \{\alpha_{1},\dots,\alpha_{n_1+n_2}\}}$ breaks up as a sum
$$\sum_{\beta_1 +\beta_2=\beta, w\in W/W_P}I^{0,X}_{\beta_1, \{\alpha_1,\dots,\alpha_{n_1},[X_{w}]\}} \ I^{0,X}_{\beta_2, \{\alpha_{n_1+1},\dots,\alpha_{n},[X_{w'}]\}}$$

These can be generalized to analogous factorization formulas for $I^{c,X}_{\beta, \{\alpha_{1},\dots,\alpha_{n},[T_0]\}}$ in case $(X,\beta, \vec{\alpha})$ satisfies the codimension c cycle condition.

The GW classes also satisfy a formula  \cite[Section 2.2.3]{KM}, analogous to what is called {\bf{Propagation of Vacua}} for vector bundles of conformal blocks.  Namely, let
 $T_0\in A^0(X)$ be the fundamental class of the space. Then  if $(X,\beta, \vec{\alpha})$ satisfies the codimension c cycle condition, then
 $$I^{c,X}_{\beta, \{\alpha_{1},\dots,\alpha_{n},[T_0]\}}=\pi_{n+1}^* I^{c,X}_{\beta, \{\alpha_{1},\dots,\alpha_{n}\}},$$ where $\pi_{n+1}:\ovop{M}_{0,n+1}\to \ovop{M}_{0,n}$ is the projection map.
 \subsubsection{Small quantum cohomology}\label{small}
Assume $X$ is a homogenous space as before.  Let $T_1,\dots,T_p$ be a basis of $A^1(X)$. Let $Z[q]=Z[q_1,\dots,q_p]$ where $q_1,\dots,q_p$ are formal variables. For $\beta\in H_2(X)$, let
 $q^{\beta}=q_1^{\beta\cdot T_1}q_2^{\beta\cdot T_2}\dots q_p^{\beta\cdot T_p}$, and set $QH^*(X)= H^*(X)\tensor \Bbb{Z}[q]$. Define a small quantum $\Bbb{Z}[q]$-algebra structure  $\star$ on $QH^*(X)$ by
 $$\alpha_1\star \alpha_2= \sum_{\gamma,\beta}q^{\beta}\langle  \alpha_1,\alpha_2,\gamma\rangle_{\beta} \gamma'$$
 here $\alpha_1,\alpha_2\in H^*(X)$ and $\beta$ runs through $H_2(X)$, and $\gamma$ runs through all Schubert cycle classes.

\subsubsection{Proof of Proposition \ref{DivisorIntersection}}\label{DivIntProof}

Let $\op{F}_{N_1,\ldots,N_4}$ be an $\op{F}$-Curve, and  $I_{\vec{\alpha}}^{\beta}(X)$ a GW-divisor on $\ovop{M}_{0,n}$.
Without loss of generality we can rename the $\alpha_i$ so that
$\{\alpha_i : i\in N_j\}=\{\alpha^j_j,\ldots,\alpha^j_{n_j}\}$,
where $n_j=|N_j|$. There is a surjective map from a product of $\ovop{M}_{0,4}$ and four copies of $\ovop{M}_{0,3}$ onto $F_{N_1,N_2,N_3,N_4}$. To compute the class of $I_{\vec{\alpha}}^{\beta}(X)$, one pulls the divisor
back to the product of the moduli spaces.  By the factorization formula,  one gets the asserted formula.

\begin{remark}The proof of Proposition \ref{HigherCodimensionIntersection} is analogous to the proof of Proposition \ref{DivisorIntersection}.
\end{remark}

\subsection{Divisor intersection simplifications}\label{Reconstruction}

In order to find classes of GW divisors, one needs to know how to find:
 $$I^{0,X}_{\beta, \{\alpha_{1},\dots,\alpha_{n}\}}\in A^0(\ovop{M}_{0,4})=\Bbb{Z} \mbox{ and }
\ I^{1,X}_{\beta, \{\alpha_{1},\dots,\alpha_{4}\}}\in \op{Pic}(\ovop{M}_{0,4})=\Bbb{Z}.$$

Often these quantities can be simplified computationally.  For example:
\begin{enumerate}
\item $I^{0,X}_{\beta, \{\alpha_{1},\dots,\alpha_{n}\}}$ always reduces to  $3$-point  GW numbers, which are  the coefficients
 of $q^{\beta}[pt]$ in the small quantum product
$$[X_{w_1}]\star [X_{w_2}]\dots \star [X_{w_n}].$$

 \item If for some $i \in \{1,\ldots, 4\}$,  the class $\alpha_i$ has codimension one, then by
\cite[Prop III, p 35]{FP},
  \begin{equation}\label{Divisor}
I^{1,X}_{\beta, \{\alpha_{1},\dots,\alpha_{4}\}}=(\alpha_4\cdot \beta) \ I^{0,X}_{\beta, \{\alpha_{1},\alpha_2,\alpha_{3}\}}\in \Bbb{Z}=\op{Pic}(\ovop{M}_{0,4}).
\end{equation}

\item If $\beta=0$, then  $I^{1,X}_{\beta, \{\alpha_{1},\ldots,\alpha_{n}\}}=0$  and $I^{0,X}_{\beta, \{\alpha_{1},\ldots,\alpha_{n}\}}$ coincides with the multiplicity of the class of a point in the  product, in cohomology $H^*(X)$, of $\alpha_1,\dots,\alpha_n$.
\end{enumerate}

\bigskip
As we next explain, another simplification of the four-point numbers
 can often be made  in terms of small quantum cohomology numbers and from identities pulled back from $\ovop{M}_{0,4}$.

 \subsubsection{}
 The formulas to be described in this section are from \cite[3.2.3, Step 2]{KM}.  We extend the definition of $I^{c,X}_{\beta, \vec{\alpha}}$ to allow for arbitrary
 $\alpha_i\in QH^*(X)=H^*(X)\tensor\Bbb{Z}[q]$ (see Section \ref{small}) by $\Bbb{Z}$-linearity (and not $\Bbb{Z}[q]$ linearity!) in $\alpha_i$, and by setting
$$I^{c,X}_{\beta, \{q^{m_1}\alpha_1, \ldots , q^{m_n}\alpha_n\}}=I^{c,X}_{\beta-\sum_i m_i, \{\alpha_1,\ldots,\alpha_n\}}.$$ Recall that the degree $|q^{\beta}\alpha|$ is $\beta\cdot c_1(T_V) + |\alpha|$ for homogeneous $\alpha\in H^*(V)$.

 \begin{proposition}\label{Recon}For $\alpha_i,\alpha_j,\alpha_k,\alpha_{\ell},\alpha_m \in QH^{*}(X)$, and homogeneous such that
 $$\sum_{x} |\alpha_x| = 1 +\beta \cdot c_1 (T_X) +\dim X,$$
 \begin{multline}
I^{1,X}_{\beta, \{\alpha_k,\alpha_{\ell},\alpha_{m},\alpha_{i}\star \alpha_j\}}
= I^{1,X}_{\beta, \{\alpha_j,\alpha_{\ell},\alpha_{m},\alpha_i \star \alpha_k\}}
+I^{1,X}_{\beta, \{\alpha_i,\alpha_k,\alpha_m,\alpha_j \star \alpha_{\ell}\}}-I^{1,X}_{\beta, \{\alpha_i,\alpha_j,\alpha_m,\alpha_{k}\star \alpha_{\ell}\}} \\
 =I^{1,X}_{\beta,\{\alpha_j,\alpha_{k},\alpha_{m},\alpha_i \star \alpha_{\ell}\}}
+I^{1,X}_{\beta,\{\alpha_i,\alpha_{\ell},\alpha_m,\alpha_j \star \alpha_k\}}-I^{1,X}_{\beta, \{\alpha_i,\alpha_j,\alpha_m,\alpha_{k}\star \alpha_{\ell}\}}.
 \end{multline}
\end{proposition}

\begin{proof}
It is easy to check that we may assume $\alpha_i,\alpha_j,\alpha_k,\alpha_{\ell},\alpha_m \in H^{*}(X)$, by writing $\alpha_i=q^{\beta_i}\alpha'_i$ etc.

We work with the contraction  morphism $\rho: \ovop{M}_{0,5}(X,\beta)\to \ovop{M}_{0,4}$.  On $\ovop{M}_{0,4}\cong \mathbb{P}^1$, one has the divisor class identities
$$\delta_{ij,k \ell}=\delta_{ik,j\ell} = \delta_{i\ell, jk}.$$
When pulled back along $\rho$, these give the identities
\begin{multline}\label{RelA}
\sum_{S}I_{3,\beta_1}^{X}(\alpha_i,\alpha_j,\gamma) \ I^{X}_{4,\beta-\beta_1}(\alpha_k,\alpha_{\ell},\alpha_{m},\gamma')
+\sum_{S}I_{4,\beta_1}^{X}(\alpha_i,\alpha_j,\alpha_m,\gamma) \ I^{X}_{3,\beta-\beta_1}(\alpha_k,\alpha_{\ell},\gamma')\\
=\sum_{S}I_{3,\beta_1}^{X}(\alpha_i,\alpha_k,\gamma) \ I^{X}_{4,\beta-\beta_1}(\alpha_j,\alpha_{\ell},\alpha_{m},\gamma')
+\sum_{S}I_{4,\beta_1}^{X}(\alpha_i,\alpha_k,\alpha_m,\gamma) \ I^{X}_{3,\beta-\beta_1}(\alpha_j,\alpha_{\ell},\gamma')\\
\sum_{S}I_{3,\beta_1}^{X}(\alpha_i,\alpha_{\ell},\gamma) \ I^{X}_{4,\beta-\beta_1}(\alpha_j,\alpha_{k},\alpha_{m},\gamma')
+\sum_{S}I_{4,\beta_1}^{X}(\alpha_i,\alpha_{\ell},\alpha_m,\gamma) \ I^{X}_{3,\beta-\beta_1}(\alpha_j,\alpha_{k},\gamma'),
\end{multline}
where $S=\{\gamma, \beta_1 \ | \ [\Delta]=\sum \gamma \otimes \gamma' \ \}$.  Using that $$\alpha_x \star \alpha_y = \sum_{\beta_1} q^{\beta_1} \langle \alpha_x, \alpha_y, \gamma \rangle \ \gamma',$$
one has
$q^{\beta_1}I^X_{4,\beta-\beta_1}(\alpha_a,\alpha_b,\alpha_c,\alpha_d)=I^X_{4,\beta}(\alpha_a,\alpha_b,\alpha_c,q^{\beta_1}\alpha_d)$.
We  may therefore rewrite Eq \ref{RelA} as

\begin{multline}\label{RelB}
 I^{X}_{4,\beta}(\alpha_k,\alpha_{\ell},\alpha_{m},\alpha_{i}\star \alpha_j)
+I_{4,\beta}^{X}(\alpha_i,\alpha_j,\alpha_m,\alpha_{k}\star \alpha_{\ell}) \\
= I^{X}_{4,\beta}(\alpha_j,\alpha_{\ell},\alpha_{m},\alpha_i \star \alpha_k)
+I_{4,\beta}^{X}(\alpha_i,\alpha_k,\alpha_m,\alpha_j \star \alpha_{\ell}) \\
 I^{X}_{4,\beta}(\alpha_j,\alpha_{k},\alpha_{m},\alpha_i \star \alpha_{\ell})
+I_{4,\beta}^{X}(\alpha_i,\alpha_{\ell},\alpha_m,\alpha_j \star \alpha_k).
\end{multline}
\end{proof}
\subsubsection{Application of Proposition \ref{Recon}}
We write a simpler version of Proposition \ref{Recon}, which when used judiciously can simplify 4-point numbers to sums of 3-point numbers.

\begin{proposition}\label{Recursive} For $\alpha_i,\alpha_j,\alpha_k,\alpha_{\ell},\alpha_m \in QH^{*}(X)$, suppose that  $\alpha_{\ell}=H_1$ is the class of a hyperplane, and $\alpha_m=H^{\star (t-1)}$, so that $\alpha_{\ell}\cdot \alpha_m = H_1 \star H^{\star, t-1}=H^{\star,t}$.   Then one can rewrite $I^{1,X}_{\beta, \{\alpha_i,\alpha_j,\alpha_k,H^{ \star t}\}}$ as:

 \begin{equation}\label{formule}
 I^{1,X}_{\beta, \{\alpha_i\star H,\alpha_k,\alpha_j,  H^{ \star(t-1)} \}} + I^{1,X}_{\beta, \{\alpha_i,H,\alpha_k, H^{ \star (t-1)}\star \alpha_j  \}} -I^{1,X}_{\beta, \{\alpha_i\star\alpha_j,\alpha_k, H,H^{ \star(t-1)} \}}
 \end{equation}

\end{proposition}

\begin{remark}
If $\alpha_2,\alpha_3,\alpha_3\in H^*(V)$ (and not in $QH^*(V))$,  by
\cite[Prop III, p 35]{FP},
$$I^{1,X}_{\beta, \{H_1,\alpha_2,\alpha_{3},\alpha_4\}}=(H_1\cdot \beta) I^{0,X}_{\beta, \{\alpha_2,\alpha_3,\alpha_4\}}.$$
and
$$I^{1,X}_{\beta, \{H_1,q^{\beta_2}\alpha_2,q^{\beta_3}\alpha_{3},q^{\beta_4}\alpha_4\}}=(H_1\cdot \beta') I^{0,X}_{\beta', \{\alpha_2,\alpha_3,\alpha_4\}}.$$
where $\beta'=\beta-\beta_2-\beta_3-\beta_4$.

Therefore in Equation \eqref{formule}, the second and third term can be computed using small quantum cohomology. The first term has $H^{\star (t-1)}$ in the last coordinate, so the exponent in $H$ has improved, and we can iterate the procedure to reduce to $t=1$.
\end{remark}

For an example of a calculation done using Proposition \ref{Recursive},  see the Section \ref{Ex3}.

\section{Projective space Examples}
The simplest case to consider is $X=\Bbb{P}^r$, and Proposition \ref{complet} links divisor classes $I^{1,\mathbb{P}^r}_{\vec{m}, d}$  on $\ovop{M}_{0,n}$ to conformal blocks divisors for type A at level 1.


\begin{proposition}\label{complet}Suppose we are given a pair $(r, \vec{m})$ such that $\sum_{i=1}^nm_i=(r+1)(d+1)$.  Then
$$I^{1,\mathbb{P}^r}_{\vec{m}, d} \equiv  c_1(\mathbb{V}(\operatorname{sl}_{r+1}, \{\omega_{m_1},\ldots, \omega_{m_n}\},1)).$$
\end{proposition}

\begin{proof}Conformal blocks divisors $c_1(\mathbb{V}(\mathfrak{g},\vec{\lambda},\ell))$ are described briefly in Section \ref{CBDivisors}.  Here we are concerned with the special case when $\mathfrak{g} =\sL_{r+1}$, and $\ell=1$.  In this case, for $\vec{\lambda}=(\lambda_1,\ldots,\lambda_n)$,  the $\lambda_i$ correspond to Young diagrams with rows $\ell=1 \ge \lambda_i^1 \ge \cdots \ge \lambda_i^r$, and the compatibility requirement is that $\sum_{i=1}^n|\lambda_i|=(r+1)(d+1)$, where $|\lambda_i|=\sum_{j=1}^r\lambda_I^j$.

It is enough to show that each divisor intersects any $\op{F}$-curve in the same degree.  By Proposition \ref{DivisorIntersection}, this amounts to proving, for any partition
$N_1 \cup \cdots \cup N_4$ of $[n]$ into nonempty subsets, if we write $\vec{m}(N_j)=\{m_i : i \in N_j\}$,  for any $\vec{a}=(a_1,a_2,a_3,a_4)$ with
$$\sum_{i \in N_j}m_i+a'_j=(r+1)(d_j+1), \mbox{ and } \sum_{i}a_i=(r+1)(d-\sum_{i}d_i+1).$$
that $I^{0,\mathbb{P}^r}_{d_j, \vec{m}(N_j) \cup a'_j}$ is proportional to $\op{rank}(\mathbb{V}(\operatorname{sl}_{r+1}, \{\omega_{m_i}: i \in N_j\} \cup \omega_{a'_j},1))$, and that
$$I^{1,\mathbb{P}^r}_{d-\sum_i d_i, \vec{a}} \equiv c_1(\mathbb{V}(\operatorname{sl}_{r+1}, \{\omega_{a_1},\ldots, \omega_{a_4}\},1)).$$

In \cite{Fakh}, Fakhruddin proved that the  level one bundles in type A always have rank one.  So it is enough to check:
\begin{enumerate}
 \item Four point classes are the same:
  $$I^{1,\mathbb{P}^r}_{\beta, \vec{a}} \equiv c_1(\mathbb{V}(\operatorname{sl}_{r+1}, \{\omega_{a_1},\ldots, \omega_{a_4}\},1)), \mbox{ where } \sum_i a_i =(r+1)(\beta+1), \mbox{ and}$$
\item Coefficients are the same: $$I^{0,\mathbb{P}^r}_{d_j, \vec{m}(N_j) \cup a'_j} = \op{Rank}\mathbb{V}(\sL_r,\{\omega_{a_i}:  i \in N_j \} \cup \omega_{a'_j},1)=1.$$
\end{enumerate}

\bigskip
To see that four point classes are the same: If one of the $a_i=0$ then the class is pulled back from $\ovop{M}_{0,3}$, and hence zero. The conformal blocks divisor is also trivial in this case. If $\beta=0$, then divisors from both theories are zero. Clearly $a_1+a_2+a_3+a_4 \leq 4r$ and hence $\beta \in \{0,1,2\}$.

We show now that if $\beta=2$, the GW divisor is zero (the same is true of the conformal blocks divisor \cite[Lemma 5.1]{Fakh}). Clearly in this case $r\geq 3$ (otherwise we will need $4$ classes in $\Bbb{P}^2$ with codimensions summing to $9$).
We want to count maps $f:(\Bbb{P}^1,p_1, p_2,p_3,p_4)\to\Bbb{P}^r$ such that $p_i$ go into specified Schubert cells (generic translates of standard cells). The image of the  conic lies
in a plane in $\Bbb{P}^r$. The space of such planes $\op{Gr}(3,r+1)$ is of dimension $3(r-2)$. The conditions imposed on this plane are at least $\sum (a_i-2)=3(r+1)-8> 3(r-2)$, hence there are
no such planes (If $a_i=3$, then we want the $\Bbb{C}^3\subset \Bbb{C}^{r+1}$ arising from the plane to meet a codimension $3$ hyperplane in $\Bbb{C}^{r+1}$ non-trivially, which imposes
one condition on the plane. Similarly if $a_i> 3$, the number of conditions imposed is $(a_i-2)$).

Finally, if $\beta=1$, we may assume all $a_i>1$, because if $a_i=1$, the GW divisor is of degree $1$ as is the conformal blocks divisor, by \cite[Lemma 5.1]{Fakh}. To look for a line of degree $1$, we can look in $\op{Gr}(2,n)$. We therefore need to intersect
the Schubert varieties corresponding to partitions $(a_1-1,0), (a_2-1,0), (a_3-1,0), (a_3-1,0)$ when $a_1+a_2+a_3+a_4-4=2(r+1)-4= 2(r-1)$ in the Grassmannian $\op{Gr}(2,r+1)$. Assuming $a_1\leq a_2\leq a_3\leq a_4$, we want the answer to be $a_1$ if $a_2+a_3\geq a_1+a_4$, and $r+1-m_4$ otherwise as is the case for the corresponding conformal blocks divisor \cite[Lemma 5.1]{Fakh}. Let $\lambda_i=a_i-1$. We do the computation in the representation ring of $\op{sl}_2$. The cup product of $V(\lambda_1)$ and $V(\lambda_4)$ is a multiplicity free string of representations
$V(\lambda_4-\lambda_1), V(\lambda_4-\lambda_1+2),\dots, V(\lambda_1+\lambda_4)$. Since $\lambda_4-\lambda_1\geq \lambda_3-\lambda_2$, the desired intersection number is $1+1/2(\lambda_4+\lambda_1-(\lambda_4-\lambda_1))=a_1$ if $\lambda_4+\lambda_1\leq \lambda_2+\lambda_3$, and equal to $1+1/2(\lambda_2+\lambda_3-(\lambda_4-\lambda_1)= 1+ (r-1-\lambda_4)= r+1-a_4$ otherwise.

Since the $\vec{m}(N_j) \cup a'_j$ satisfy the $c=0$ co-cycle condition,  $I^{0,\mathbb{P}^r}_{d_j, \vec{m}(N_j) \cup a'_j}$ can be computed with small quantum cohomology numbers for $\Bbb{P}^r$ and are easily seen to be $1$.
The ranks of the conformal blocks divisors in type A at level one are one \cite{Fakh}.

The proof of Proposition \ref{complet} is now complete.
\end{proof}

In  \cite{G}, it was shown for $\op{S}_n$-invariant divisors $c_1(\mathbb{V}(\sL_{r+1},\vec{\lambda},1))$, (and then later for general divisors in  \cite{GG}),  in case $\sum_{i=1}|\lambda_i|=(r+1)(d+1)$, that the divisors $c_1(\mathbb{V}(\sL_{r+1},\vec{\lambda},1))$ give maps to moduli spaces that generically parametrize configurations of weighted points that lie on a Veronese curve of degree $d$ in $\mathbb{P}^d$.  The statement of the result in Proposition \ref{complet} is a priori different, as it refers generically
 to maps of $\mathbb{P}^1$ to $\mathbb{P}^{r}$.

\subsection{Examples of higher codimension cycles that pushforward to divisors}\label{PushforwardExamples}
Suppose $\sum m_i\equiv c-1 \pmod{r+1}$. Then, $I^{c,\mathbb{P}^r}_{\vec{m}, d}$ is a codimension $c$ rationally strongly basepoint free cycle on $\ovop{M}_{0,n}$ with $(d+1)(r+1)+ c-1=\sum m_i$. Consider the forgetful map $\ovop{M}_{0,n}\to \ovop{M}_{0,n-1}$. We wish to determine/study $\eta_* I^{c,\mathbb{P}^r}_{\vec{m}, d}$ a basepoint free divisor on $\ovop{M}_{0,n-c+1}$.

We consider here two explicit example in case $c=2$. If $m_n=1$, then this push forward coincides with $I^{1,\mathbb{P}^r}_{\vec{m}',d}$ on $\ovop{M}_{0,n-1}$ where $\vec{m}'=(m_1,\dots,m_{n-1})$.  We therefore have selected examples for which $m_n>1$.

\subsubsection{Pushforward of $Z=I^{2,\mathbb{P}^3}_{\{H_1,H_2^6\}, 2}$ from $\ovop{M}_{0,7}$ to a basepoint free divisor on $\ovop{M}_{0,6}$}

Let $\pi: \ovop{M}_{0,7}\to \ovop{M}_{0,6}$ be the morphism which drops the $7$th marked point.  To determine the class of $\pi_*(I^{2,\mathbb{P}^3}_{\{H_1,H_2^6\}, 2})$ on $\ovop{M}_{0,6}$, we intersect $I^{2,\mathbb{P}^3}_{\{H_1,H_2^6\}, 2}$ with the pullback $\pi^*(F)$, where $F$ runs over the set of curves dual  to the nonadjacent basis.  This is the dual basis given by the boundary curves:
\begin{multline}
\{F_{1,2,3,456},F_{1,4 ,23 ,56},F_{1, 5,6 ,234}, F_{2, 3, 4, 156},  F_{2, 5, 16, 34}, F_{1, 2, 6, 345},\\
F_{3, 4, 5, 126}, F_{3, 6, 12, 45}, F_{4, 5, 6, 123},
F_{3, 4, 12, 56}, F_{5, 6, 12, 34}, F_{1, 2, 34, 56}, \\
(F_{5, 6, 13, 24} + F_{1, 2, 3, 456}+F_{2, 3, 4, 156} - F_{2, 3, 16, 45}), F_{2, 3, 16, 45},
F_{1, 6, 23, 45}, F_{4, 5, 16, 23}\}.
\end{multline}

Because of the symmetry of the Schubert classes used to define $I^{2,\mathbb{P}^3}_{\{H_1,H_2^6\}, 2}$, we can relabel these curves depending on where the first point appears:
$$\{A,B, A, C, D, A, C, D, C, D, D, B,  (A+C), D, B, D\},$$
where

\begin{multicols}{2}
\begin{itemize}
\item $A=Z\cdot \pi^*(F_{\{p_1\},\{x\},\{x\},\{x, x, x\}})$;
\item $B=Z\cdot \pi^*(F_{\{p_1\},\{x\} ,\{x, x\}\{x, x\}})$;
\item $C=Z\cdot \pi^*(F_{\{x\},\{x\},\{x\},\{p_1,x,x\}})$; and
\item $D=Z\cdot \pi^*(F_{\{x\}, \{x\}, \{p_1,x\}, \{x,x\}})$.
\end{itemize}
    \end{multicols}

We calculate, for every surface above the intersection of $Z$ with its three components.  We'll check below that
$Z \cdot \pi^*(F_{\{p_1\},\{x\},\{x\},\{x,x,x\}})=Z\cdot (Z^A_1+Z^A_2+Z^A_3) = 2+0+2.$
Here $Z^A_1\cong \ovop{M}_{0,5}\times \ovop{M}_{0,3}\times \ovop{M}_{0,3}$, with the $7$th point on the first component, $Z^A_2 \cong \ovop{M}_{0,4}\times \ovop{M}_{0,4}\times \ovop{M}_{0,3}$, with the $7$th point on the second component, and $Z^A_3\cong \ovop{M}_{0,4}\times \ovop{M}_{0,3}\times \ovop{M}_{0,4}$,
with the $7$th point on the third component.  To compute the intersection,  using factorization, we determine what Schubert classes can be used as attaching data.

In this example, on the factor isomorphic to $\ovop{M}_{0,5}$, we are given Schubert classes  $\{H_1, H_2, H_2, H_2\}$  at four points, and we let $\alpha_1$ be the class at the $5$-th {\em{attaching}} point.  On the second factor isomorphic to $\ovop{M}_{0,3}$, we have one given class $H_2$ and two
attaching classes $\alpha_1^*$ and $\alpha_2$, and on the third factor isomorphic to $\ovop{M}_{0,3}$ there are two given classes, $\{H_2,H_2\}$, and the attaching class $\alpha_2*$.  Since the total degree in $Z$ is $2$, when restricted to the three factors, $Z$ will decompose as a product of $GW$ classes one of which has degree $0$.  If the degree zero class is on the factor isomorphic to $\ovop{M}_{0,5}$, the result will give a zero intersection.  There is one choice for  a nonzero intersection given by letting
$$\alpha_1=H_2, \alpha_1^*=H_1, \alpha_2=H_0, \alpha_2^*=H_3, \mbox{ which gives } Z \cdot Z^A_1 = 2.$$
There is no choice giving a nonzero intersection of $Z$ with $Z^A_2$.  Using an argument similar to the first, one can check that the intersection of $Z$ with $Z^A_3$ is $2$.

We shall see that $Z\cdot \pi^* B=Z \cdot (Z^B_1+Z^B_2+Z^B_3)=0$.  Here again on $Z^B_i$ the the 7th point is on the $i$-th component, but this time there are two attaching points on
the first factor in each case.  For instance
Here $Z^B_1\cong \ovop{M}_{0,5}\times \ovop{M}_{0,3}\times \ovop{M}_{0,3}$, with the $7$th point on the first component, and we are given Schubert classes
$\{H_1, H_2, H_2\}$  at three points, and let $\alpha_1$ and $\alpha_2$ be the classes at the $4$th and $5$-th points which are the points attaching this factor two the other two.  The other two factors constrain $\alpha_1=\alpha_2=H_0$, which forces the degree of $Z$ on this surface to be zero.
This same phenomena happens to at least one of the attaching points for each of the cases $Z^B_2$ and $Z^B_3$ and the total degree is $0$. One has that $C=6$, and that $D=0+0+2=2$.
$$[A,B, A, C, D, A, C, D, C, D, D, B,  (A+C), D, B, D]=
2[2,0, 2, 3, 1, 2, 3, 1, 2, 1, 1, 0,  5, 1, 0, 1].$$

One can check (using a program such as LRS \cite{LRS}, as we did) that this intersects $15$ F-Curves on $\ovop{M}_{0,6}$ in degree zero.  The face determined by these is 4 dimensional, and and $Z$ can be expressed as a combination of the following extremal rays that span this face:
\begin{multicols}{2}
\begin{enumerate}
\item $[1, 0, 1, 1, 0, 1, 1, 0, 1, 0, 0, 0, 2, 0, 0, 0]$;
\item $[0, 0, 0, 1, 0, 0, 1, 0, 0, 1, 0, 0, 1, 0, 0, 0]$;
\item $[0, 0, 0, 0, 1, 0, 0, 1, 0, 0, 1, 0, 0, 1, 0, 0]$;
\item $[1, 0, 1, 1, 0, 1, 1, 0, 1, 0, 0, 0, 2, 0, 0, 1]$.
\end{enumerate}
    \end{multicols}

The first extremal ray listed is the same as the first ray $R_1$ on Swinarski's list, where he classifies the $\op{S}_6$-equivalence classes of the $3,190$ extremal rays of the nef cone for $\ovop{M}_{0,6}$.  The second and third divisors listed generate rays in the class for $R_6$ on his list, and the last represents the class of $R_5$.   So while these 4 rays span a 4-dimensional extremal face, they are of three different types on Swinarski's list.

\subsubsection{Pushforward of $Z=I^{2,\mathbb{P}^3}_{\{H_1^4,H_3^3\}, 2}$ from $\ovop{M}_{0,7}$ to  $\ovop{M}_{0,6}$}

Let $\pi: \ovop{M}_{0,7}\to \ovop{M}_{0,6}$ be the morphism which drops the $7$th marked point.  To determine the class of $\pi_*(I^{2,\mathbb{P}^3}_{\{H_1,H_2^6\}, 2})$ on $\ovop{M}_{0,6}$, we intersect $I^{2,\mathbb{P}^3}_{\{H_1,H_2^6\}, 2}$ with the pullback $\pi^*(F)$, where $F$ runs over the set of curves dual  to the nonadjacent basis, and because of the symmetry of the Schubert classes used to define $Z$, we only need to keep track of where the 5th and 6th points are.  The class is given by \begin{multline}[A,B,C, A,  D, E,E, D, C,B, F, B, (F + 2A- G), G,D, D]=[1,1,2,1, 1, 2,2, 1, 2,1, 1, 1, 3,0,1, 1]\\
=[0,1,1,1, 0, 1,1, 1, 0,1, 0, 0, 1,0,1, 0]+[1,0,1,0, 1, 1,1, 0, 2,0, 1, 1, 2,0,0, 1]=R_5+R_{16},
\end{multline}
where for $x \in \{p_1,p_2,p_3,p_4 \}$, and $y \in \{p_5,p_6\}$
one has
\begin{multicols}{2}
\begin{itemize}
\item $A=Z\cdot \pi^*F_{\{x\},\{x\},\{x\},\{x,y,y\}}=1$;
\item $B=Z\cdot \pi^*F_{\{x\},\{x\} ,\{x,x\},\{y,y\}}=1$;
\item $C=Z\cdot \pi^*F_{\{x\}, \{y\},\{y\} ,\{x,x,x\}}=2$;
\item $D=Z\cdot \pi^*F_{\{x\}, \{y\}, \{x,y\}, \{x,x\}}=1$;
\item $E=Z\cdot \pi^*F_{\{x\},\{x\},\{y\}, \{x,x,y\}}=2$;
\item $F=Z\cdot \pi^*F_{\{y\}, \{y\},\{x,x\},\{x,x\}}=1$;
\item $G=Z\cdot \pi^*F_{\{x\},\{x\},\{x,y\},\{x,y\}}=0$;
\end{itemize}
    \end{multicols}
and by $R_5$ and $R_{16}$ we mean the 5th and 16th rays on the list in Swinarski's enumeration of equivalence classes of extremal rays of the nef cone of $\ovop{M}_{0,6}$ in \cite{Swin}.

The divisor  $\pi_*Z$ contracts 12 F-Curves on $\ovop{M}_{0,6}$, and there are $23$ extremal rays of the nef cone $\op{Nef}(\ovop{M}_{0,6})$ that also contract those F curves. In particular,  $\pi_*Z$ lies on the
face spanned by the 23 extremal rays. Using LRS \cite{LRS}, one can check that as a cone, this face is 7 dimensional.  In particular,  a generic element on this face would be described by an effective combination of 7 divisors.  But as one can see above,  this divisor is a combination of two extremal rays.
\begin{remark}\label{Interesting}
While $R_5$  (as well as $R_1$ and $R_6$ from the previous example)  are known to be spanned by conformal blocks divisors, there is no known conformal blocks divisor that spans $R_{16}$.  In other words, while  Proposition \ref{complet} links conformal blocks divisors to GW divisors classes for $X=\mathbb{P}^r$, there does not appear to be a link between conformal blocks divisors and the divisors one obtains by pushing forward GW classes of higher codimension  for $X=\mathbb{P}^r$.
\end{remark}

\subsubsection{$D^2$}
Suppose $\sum m_i=(d+1)(r+1)$.  Then $D=I^{1,\mathbb{P}^r}_{\vec{m}, d}$ is a  rationally strongly basepoint free divisor, and $D^2$ is a rationally strongly base point free codimension two cycle on $\ovop{M}_{0,n}$.  Hence by Lemma \ref{SBPFprops}, the pushforward
 $\pi_*(D^2)$ is a basepoint free divisor class on $\ovop{M}_{0,n-1}$, where $\pi:\ovop{M}_{0,n}\to \ovop{M}_{0,n-1}$ is any of the $n$ point dropping maps.  Now the boundary cycles  span $A^1(\ovop{M}_{0,n})$ (by \cite{KeelThesis}), and one can use the intersection formulas there to compute $D^2$, the pushforward  of $D^2$ can be computed by the formulas in \cite{AC}.

\section{Quadrics and GW divisors}
Let $X=Q_r$ be a smooth projective quadric of even dimension $r=2m \ge 4$ or of odd dimension $r=2m+1 \ge 1$ given by a nondegenerate quadratic form on a vector space $V$, of dimension $r+2$ over a field $F$, so  $X \subset \mathbb{P}(V)$. Let $H=H_1\in A^1(X)$ the pullback of the hyperplane class in $A^1(\mathbb{P}(V))$. We let $H_i=H_1^i$ (the $i$-fold cup product in ordinary cohomology)for $i\in[1,r]$ and $H_0=1$.

The degree of the canonical bundle of a smooth projective quadric $Q_r$ is $-r$.  So for
$(X,d,\vec{\alpha})$ to satisfy the codimension $c$ cycle condition we must have that
$$\sum_{i}|\alpha_i| = c+ r(d+1),$$
In case $r=2m+1$ is odd:
\begin{itemize}
\item Let $W$ be a maximal totally isotropic subspace so $\mathbb{P}(W) \subset X$, $\dim\mathbb{P}(W)=m$.  For any integer $i \in [0,m]$, let $L_i\in A_i(X)$ be the class of an $i$ dimensional subspace of $\mathbb{P}(W)$.  Then the total Chow ring of $X$ is free  with basis  $\{H_i, L_i \ | \ i \in [0,m]\}$. Note that $H_{m+1}=2L_{m}$ for $i\in[1,m+1]$ and  $H\cdot L_i=L_{i-1}$ for any $i \in [1,m]$.
\item As a basis of the rational cohomology we take
$$1=H_0, H_1, \ldots, H_r.$$
\end{itemize}
In case $r=2m$ is even:
\begin{itemize}
\item The space of maximal isotropic subspaces of $V$ has two components. Let $W_1, W_2$ be representatives in each component. Now  $\mathbb{P}(W_a) \subset X$, $a=1,2$ and let $\xi_1,\xi_2\in A_m(X)$ be their cycle classes. For $i\in[0,m-1]$, let $L_i\in A_i(X)$ be the cycle class of an $i$ dimensional subspace of $\Bbb{P}(W_1)$ (note that we get the same cycle class if $W_1$ is replaced by $W_2$ here).  The total Chow ring of $X$ is free with basis $H_0=1$, $H_1$, $\ldots$,$H_{m-1}$, $\xi_1$, $\xi_2$, $L_{m-1}\ldots$, $L_0$. We also have  $H\cdot L_i=L_{i-1}$ for any $i \in [1,m-1]$, $H_m=\xi_1+\xi_2$ and $H\cdot\xi_a= L_{m-1}$ for $a=1,2$, so that $H_{m+1}=2L_{m-1}$.
\item For even dimensional quadrics,  as a basis of the rational cohomology we take
$$1=H_0,  H_1, \ldots, H_{m-1}, \xi_1, \xi_2, H_{m+1},\ldots, H_r.$$
\end{itemize}

In both cases, for $X=Q_r$ (even or odd), $H^{\star j}=H_j$  if $j<r$. If $j=r$, then
$H^{\star j}$ equals $H_j$ plus a multiple of $q$ times the identity in cohomology. But a four point number  with one of the terms equalling identity
in cohomology is zero. Therefore we may always write
$(H_j,y_1,y_2,y_3)_{\beta}=(H^{\star,j},y_1,y_2,y_3)_{\beta}$ and apply Proposition \ref{Recursive} to simplify intersection formulas.

The cohomology of an even quadric is generated by the hyperplane class except in the  middle dimension.  But we cannot have a $4$ point number with all four terms in the middle dimension, since the codimensions need to add up to $1$ mod $r$. Therefore $4$ point numbers for even quadrics are computable with these methods.
For proofs of these statements and more on quadric hypersurfaces see \cite[Part 3]{EKM}.

To compute classes, we determine certain facts about the quantum cohomology of $X=Q_r$.

\begin{lemma}\label{basic}

$I^{0,Q_r}_{1, \{H_1,H_r, H_{r-1} \}}=4$

\end{lemma}

\begin{lemma}\label{mult}
$$H_i \star H_j = \left\{
\begin{matrix}
H_{i+j} & \text{if } i+j < r \\
H_r+2qH_0 & \text{if } i+j = r \\
4qH_{\ell} & \text{if } i+j = r + \ell, \text{ with } i<r \text{ and  } j<r\\
2qH_i  & \text{if } i<r, \text{ and } j=r \\
4q^2 H_0 &  \text{if } i= j=r
\end{matrix}
\right.$$
\end{lemma}

\begin{remark}The formulas in Lemmas \ref{basic} and \ref{mult} hold for $r$ both even and odd.  Formulas specific to the even case are given in Section \ref{EvenQuadrics}.
\end{remark}

\begin{proof}(of Lemma \ref{basic})
For the first assertion, we need to count lines in the quadric $Q$ which pass through a point $P$, and a fixed line $L$ in the quadric. Clearly four times this count gives us the  desired answer since $H_r$ and $H_{r-1}$ are twice the classes of a point and a line respectively. Consider the projective space spanned by the point $P$ and the fixed line $L$, a $\Bbb{P}^2$. The quadric, restricted to this $\Bbb{P}^2$ splits as a product $Q=LL'$ since it contains $L$, we may assume that $L'$ passes through the point $P$ ($P\not\in L$), and hence $L'$ is the unique line we are looking for (it certainly meets $L$).

\end{proof}

\begin{proof} (of Lemma \ref{mult}) For odd quadrics, one can show that $H^{\star i}=H_i$ if $i\leq r-1$, and $H^{\star r}=H\star H_{r-1} =H_r + 2q\cdot 1$, since the dual of $1$ is $\frac{1}{2}H_r$.  $H^{\star r+1}= 2qH + H\star H_r=4qH,$ since the dual of $H$ is $\frac{1}{2}H_{r-1}$.  For even quadrics we need the action of orthogonal group on the space of $m+1$ isotropic subspaces of $\Bbb{C}^{2m+2}$ has two components. The dimension of intersection of two subspaces in the same connected component is constant modulo two. Since the three point number $\langle H,H_r, H_{r-1}\rangle_1$ is  equal to $4$,  the
multiplication rules for $H_i\star H_j$ are the same.

\end{proof}

\subsection{Examples for odd quadrics}

 In Section \ref{Ex3} we illustrate the use of Proposition \ref{Recursive}, which
 simplifies 4-point numbers to sums of 3-point numbers.  In Section \ref{Extremality} we show how by using the formulas for the intersections of divisors and curves, one can find GW divisors that are extremal in the nef cone by selecting parameters that will guarantee the divisor contracts boundary curves.  In Section \ref{ExtremalExamples} we have calculated examples of extremal rays of the nef cone $\op{Nef}^1(\ovop{M}_{0,6})$ that come from GW divisors from odd quadrics.

\subsubsection{Examples of use of Proposition \eqref{Recursive} for odd quadrics}\label{Ex3}
Let $X=Q_3$, and we evaluate
$$M=I^{1,X}_{2, \{L_0,L_0,L_1,L_1\}}=1$$
 (as reported in \cite{FP}, page 44 $N_{2,2}=1$). Note that
$L_0=\frac{1}{2}H_3= \frac{1}{2}H^{\star 3}-q\cdot 1$, $L_1=\frac{1}{2}H^{\star 2}=\frac{1}{2}H_2$. Therefore,
$M=\frac{1}{16}I^{1,X}_{2, \{H_3,H_3,H_2,H_2\}}$. We will show that $I^{1,X}_{2, \{H_3,H_3,H_2,H_2\}}=16$:

Using Proposition \eqref{Recursive}, and writing $H_2=H\star H$, we get
$I^{1,X}_{2, \{H_3,H_3,H_2,H_2\}}=A+B-C$, where
\begin{itemize}
\item $A=I^{1,X}_{2, \{H_3\star H,H_3,H_2,H\}}=2I^{1,X}_{1, \{H,H_3,H_2,H\}}=2I^{0,X}_{1, \{H_3,H_2,H\}}$. Now
$I^{0,X}_{1, \{H_3,H_2,H\}}$ equals coefficient of $qT_0$  in $H_3\star H_2\star H=2qH_2\star H=2q (H_3+2q)=2q(2T_0)+4q^2$. Therefore $A=8$.
\item $B=I^{1,X}_{2, \{H_3, H,H_2,H_3\star H\}}= I^{1,X}_{2, \{H_3, H,H_2,2q H\}}=2I^{1,X}_{1, \{H_3, H,H_2,H\}}=2I^{0,X}_{2, \{H_3,H_2,H\}}$. Now $H_3\star H_2\star H=2qH_2\star H=2q (H_3+2qH_0)=2q(2T_0+2qH_0)$. Therefore $B=8$.
\item $C=I^{1,X}_{2, \{H_3\star H_3,H_2,H,H\}}=0$.
\end{itemize}

\subsubsection{Extremality results}\label{Extremality}
As the following results show,
 is straight forward to design divisors $I^{1,Q_r}_{d, \vec{a}}$, for $Q_r\subset \mathbb{P}^{r+1}$ an odd quadric, that lie on extremal faces of the nef cone.
\begin{proposition}\label{Contract1}
Let $Q_r\subset \mathbb{P}^{r+1}$ be an odd quadric, and $I^{1,Q_r}_{d, \vec{a}}$ a GW-divisor.  If there is an index $i \in[n]$ such that $a_i=r$ and $J\subset [n]\setminus i$, such that for all $j\in J$, $1\le a_j \le r$, and $\sum_{j\in J}a_j=r$, then $I^{1,Q_r}_{d, \vec{a}}$ contracts
any $F$-curve of the form $F_{I, A, B,C}$, for $I=J\cup \{i\}$, and hence lies on a face of the nef cone.
\end{proposition}
\begin{proof}
 Let $r=2m-1$, and two indices $i$ and $j \in [n]$ with $a_i=a_j =r$, then the divisor  $I^{1,Q_r}_{d, \vec{a}}$ will kill any $F$-Curve of the form $F_{I, A,B,C}$ where $I=\{a_i,a_j\}$, since $H_r\star H_r=4qH_0$, and $4qI^{1,Q_r}_{d, \{H_0,\alpha_1,\alpha_2,\alpha_3\}}=0$ for all possible $\alpha_1,\alpha_2,\alpha_3$ under consideration.
More generally if there is an index $i \in [n]$ such that $a_i=H_r$, and $J\subset [n]\setminus i$, such that for all $j\in J$, $1\le a_j \le r$, and $\sum_{j\in J}a_j=r$, one has $a_i=H_i$, and the star product of classes in $J$ is $H_r$, so the star product of classes in $I=J\cup \{i\}$ is $H_r\star H_r=4qH_0$,  and the result follows.
\end{proof}

\begin{proposition}\label{Contract2}
Let $Q_r\subset \mathbb{P}^{r+1}$ be an odd quadric, and $I^{1,Q_r}_{d, \vec{a}}$ a GW-divisor with $d\le 4$.
\begin{enumerate}
\item[$d=1$] If there are indices $a_1$ and $a_2$ such that $a_1+a_2 >r$, then $I^{1,Q_r}_{1, \vec{a}}$ will kill any $F$-Curve of the form $F_{A,B,C,D}$ where $\{a_1,a_2\} \subset A$.
\item[$d=2$] If there are indices $a_1$, $a_2$, $b_1$, and $b_2$ such that $a_1+a_2 >r$, and $b_1+b_2 >r$ then $I^{1,Q_r}_{2, \{\vec{a}}$ will kill any $F$-Curve of the form $F_{A,B,C,D}$ where $\{a_1,a_2\} \subset A$ and  $\{b_1,b_2\}\subset B$.
\item[$d=3$] If there are indices $a_1$, $a_2$, $b_1$, $b_2$,  $c_1$, and $c_2$, such that $a_1+a_2 >r$,  $b_1+b_2 >r$ and  $c_1+c_2 >r$ then $I^{1,Q_r}_{3, \{\vec{a}}$ will kill any $F$-Curve of the form $F_{A,B,C,D}$ where
$\{a_1,a_2\} \subset A$, $\{b_1,b_2\}\subset B$, and $\{c_1,c_2\}\subset C$.
\item[$d=4$] If there are indices $a_1$, $a_2$, $b_1$, $b_2$,  $c_1$, $c_2$, $d_1$, and  $d_2$,such that $a_1+a_2 >r$,  $b_1+b_2 >r$,  $c_1+c_2 >r$, and $d_1+d_2 >r$ then $I^{1,Q_r}_{4, \{\vec{a}}$ will kill any $F$-Curve of the form $F_{A,B,C,D}$ where
$\{a_1,a_2\} \subset A$, $\{b_1,b_2\}\subset B$, $\{c_1,c_2\}\subset C$, and  $\{d_1,d_2\}\subset C$.
\end{enumerate}
\end{proposition}

\begin{proof}
Intersections on the leg bring the degree down by one, leaving the spine at degree zero.
\end{proof}

\subsubsection{Examples of classes from odd quadrics}\label{ExtremalExamples}
Here we consider two GW divisors in $\op{Pic}(\ovop{M}_{0,6})$:
 \begin{enumerate}
\item $I^{1,Q_r}_{4, \{1,r,r,r,r,r\}}=16 \ R_{1}$;
\item $I^{1,Q_r}_{4, \{i,j,r,r,r,r\}}= 8  R_{10}$, where $1< i\le j$, and $i+j=r+1$;
\end{enumerate}

To compute the class of  $I^{1,Q_r}_{4, \{1,r,r,r,r,r\}}$, we intersect with dual curves, and we'll see that
$$\frac{1}{16}G^4_{1,r,r,r,r,r}=\delta_{13}+\delta_{15}+\delta_{24}+\delta_{26}+\delta_{35}+\delta_{46}+2\delta_{135}=R_{1},$$
where $R_{1}$ is the first ray on the list of extremal rays of $\op{Nef}(\ovop{M}_{0,6})$ listed in \cite{Swin}.

\bigskip

Let $4 I^{1,Q_r}_{2, \{H_1, H_r, H_r, H_r\}} =\alpha$.  We have
$I^{0,Q_r}_{d, \{H_a,H_b,H_c\}}$ is $\frac{1}{2}$ times the coefficient of $q^{d}H_{r-c}$ in $H_a\star H_b$,
since the cohomology class of the diagonal has the term $\frac{1}{2}H_c\tensor H_{r-c}$.
Using Lemmas \ref{basic}, and \ref{mult}
 $$4 I^{1,Q_r}_{2, \{H_1, H_r, H_r, H_r\}} =4 \cdot 2 \ I^{0,Q_r}_{2, \{H_r, H_r, H_r\}}=4 \cdot 2 \cdot 2=16.$$

In the following (and the other examples here), to save space, once we put the classes into the place holders for points in the F-Curves, we mean we are intersecting the divisors with the F-curves.
\begin{multicols}{2}
\begin{itemize}
\item[(13)] $D\cdot F_{1,2,3,456} = F_{1,r,r, r^3}  =16$;
\item[(14)]  $D\cdot F_{1,4 ,23 ,56}=  F_{1,r ,r^2 ,r^2}=0$;
\item[(15)] $D\cdot F_{1, 5,6 ,234} =F_{1, r,r ,r^3} =16$;
\item[(24)] $D\cdot F_{2, 3, 4, 156} = F_{r, r, r, \{1r^2\}} =16$;
\item[(25)]  $D\cdot F_{2, 5, 16, 34}= F_{r, r, \{1r\}, r^2}= 0 $;
\item[(26)]  $D \cdot F_{1, 2, 6, 345} =F_{1, r, r, r^3} =16$;
\item[(35)]  $D \cdot F_{3, 4, 5, 126} =  F_{r, r, r, \{1r^2\}} =16$;
\item[(36)]  $D \cdot F_{3, 6, 12, 45} =  F_{r, r, \{1r\}, r^2} =0 $;
\item[(46)]  $D \cdot F_{4, 5, 6, 123}= F_{r, r, r, \{1r^2\}}=16$;
\item[(124)]  $D \cdot F_{3, 4, 12, 56} =  F_{r, r, \{1r\}, r^2} = 0 $;
\item[(124)]  $D \cdot F_{5, 6, 12, 34} =  F_{r, r, \{1r\}, r^2} =0 $;
\item[(134)]  $D \cdot F_{1, 2, 34, 56} =  F_{1,r ,r^2 ,r^2}= 0$;
\item[(136)]  $D \cdot F_{2, 3, 16, 45} = F_{r, r, \{1r\}, r^2} = 0 $;
\item[(145)]  $D \cdot F_{1, 6, 23, 45} =  F_{1,r ,r^2 ,r^2}=0$;
\item[(146)]  $D \cdot  F_{4, 5, 16, 23} =  F_{r, r, \{1r\}, r^2} 0 $;
\end{itemize}
\end{multicols}

For the coefficient of $\delta_{135}$:
 \begin{multline}
D \cdot (F_{5, 6, 13, 24} + F_{1, 2, 3, 456}+F_{2, 3, 4, 156} - F_{2, 3, 16, 45})) \\
=F_{r, r, \{1r\}, r^2} + F_{1,r,r, r^3}  +  F_{r, r, r, \{1r^2\}}-F_{r, r, \{1r\}, r^2} =2(16)
\end{multline}

 The expression for $I^{1,Q_r}_{2, \{H_1, H_r, H_r, H_r\}} $ is $\op{S}_6$-invariant, even though the choice of Schubert classes defining it is not.
There are relations coming from the fact that $\ovop{M}_{0,4}\cong \mathbb{P}^1$, giving the extra symmetry in the class.

This class was identified in \cite{Swin} to be spanned by an $\sL_2$-conformal blocks divisor at level one.  By scaling identities  for level one bundles in type A (which have rank one), we know that $R_1$ is proportional to
$$c_1(\mathbb{V}(\sL_{2\ell},\{\omega_{\ell}^6\},1))=\frac{1}{\ell} c_1(\mathbb{V}(\sL_{2},\{\omega_1^6\},1))
=c_1(\mathbb{V}(\sL_{2},\{(\ell\omega_1)^6\},\ell)), \ \ \ell \ge 1.$$

There is a curious and certainly tenuous potential relationship  between odd quadric GW divisors and conformal blocks divisors for $\sL_2$:  The automorphism group of an odd quadric $Q_{2m+1}$ is $SO_{2(m+1)+1}(\mathbb{C})$.  The Langland's dual group to $SO_{2(m+1)+1}(\mathbb{C})$ is $Sp_{2(m+1)}(\mathbb{C})$, which has associated Lie algebra $\mathfrak{s}\mathfrak{p}_{2(m+1)}$.  Fakhruddin proved in \cite{Fakh}, that the level one type $C_{\ell}$ bundles with four points are the same as the level $\ell$ bundles for $\sL_2$.  So perhaps there is a general level 1 identity between the GW divisors for odd quadrics and the CB divisors for $\sL_2$ at level one.

\subsubsection{$G^4_{i,j,r,r,r,r}$, where $1< i\le j$, and $i+j=r+1$}

 To compute the class, set $4 \ I^{1,Q_r}_{2, \{H_i,H_j,H_r,H_r\}}=\alpha$ and
$4 \ I^{0,Q_r}_{2, \{H_r,H_r,H_r\}}=\beta$.  Below in Lemma \ref{alphabeta}, we see that $\alpha=\beta=8$, so that
$$\frac{1}{8}G^4_{i,j,r,r,r,r}=[1:0:1:1:0:1:1:0:1:0:0:0:3:0:0:1] =R_{10},$$
where $R_{10}$ is the $10$-th ray\footnote{We know $\op{R}_{10}= \rho c_1(\mathbb{V}(\sL_{r+1},\{\omega_1^2, \omega_i,  \omega_{r+1-i}, \omega_r^2\},1))$, for $1<i\le \frac{(r+1)}{2}$, and some positive rational $\rho$.} on the list of extremal rays of the nef cone of $\ovop{M}_{0,6}$ in \cite{Swin}.

\begin{multicols}{2}
\begin{enumerate}
\item[(13)] $D\cdot F_{1,2,3,456} = F_{i,j,r,r^3} = \alpha$;
\item[(14)]  $D\cdot F_{1,4 ,23 ,56}= F_{i,r ,\{jr\} ,r^2}=0$;
\item[(15)] $D\cdot F_{1, 5,6 ,234} =  F_{i, r,r ,\{jr^2\}} =\alpha$;
\item[(24)] $D\cdot F_{2, 3, 4, 156} = F_{j, r, r, \{ir^2\}} =\alpha$;
\item[(25)] $D\cdot F_{2, 5, 16, 34}=  F_{j, r, \{ir\}, r^2}=0$;
\item[(26)] $D \cdot F_{1, 2, 6, 345} = F_{i, j, r, r^3}=\alpha $;
\item[(35)] $D \cdot F_{3, 4, 5, 126} = F_{r, r, r, \{ijr\}} =\beta$;
\item[(36)] $D \cdot F_{3, 6, 12, 45} =  F_{r, r, \{ij\}, r^2} =0$;
\item[(46)] $D \cdot F_{4, 5, 6, 123}= F_{r, r, r, \{ijr\}}=\beta$;
\item[(124)] $D \cdot F_{3, 4, 12, 56} = F_{r, r, \{ij\}, r^2} =0$;
\item[(125)] $D \cdot F_{5, 6, 12, 34} =  F_{r, r, \{ij\}, r^2} =0$;
\item[(134)] $D \cdot F_{1, 2, 34, 56} = F_{i, j, r^2, r^2} =0$;
\item[(136)] $D \cdot F_{2, 3, 16, 45} = F_{j, r, \{ir\}, r^2} = 0$;
\item[(145)] $D \cdot F_{1, 6, 23, 45} = F_{i, r, \{jr\}, r^2} = 0$;
\item[(146)]$D \cdot  F_{4, 5, 16, 23} =  F_{r, r, \{ir\}, \{jr\}} =\alpha$.
\end{enumerate}
\end{multicols}

and the coefficient of $\delta_{135}$ is given by
 \begin{multline}
D \cdot (F_{5, 6, 13, 24} + F_{1, 2, 3, 456}+F_{2, 3, 4, 156} - F_{2, 3, 16, 45})) \\
=(F_{r, r, \{ir\},  \{jr\}} + F_{i, j, r, r^3}+F_{j, r, r, \{ir^2\}} - F_{j, r, \{ir\}, r^2})) \\
=4(H_r,H_r,H_i,H_j)_2+4I^{1,Q_r}_{2, \{H_i,H_j,H_r,H_r\}}+4(H_i,H_r,H_r,H_i)_2-8(H_j,H_r,H_i,H_0)_1=3\alpha
\end{multline}

\begin{lemma}\label{alphabeta}
$\alpha=\beta=8$.
\end{lemma}
\begin{proof}
To get $\beta$, we compute $I^{0,Q_r}_{2, \{H_r,H_r,H_r\}}=\frac{1}{2}$ times the coefficient of $q^{2}H_{0}$ in $H_r\star H_r=4q^2H_0$.  So
$I^{0,Q_r}_{2, \{H_r,H_r,H_r\}}=2$, and $\beta=8$.

To get $\alpha$, we compute $I^{1,Q_r}_{2, \{H_i,H_j,H_r,H_r\}}$ using Proposition \ref{Recursive}.  We write $I^{1,Q_r}_{2, \{H_i,H_j,H_r,H_r\}}$ as
\begin{multline}
=2\left(I^{0,Q_r}_{2, \{H_r,H_i,H_r \star H_j\}}- I^{0,Q_r}_{2, \{H_i,H_{j-1}, H_r \star H_r\}}  \right)+I^{1,Q_r}_{2, \{H_r,H_i,H_{j-1},H_r\star H_1\}}\\
=2 \cdot 2 I^{0,Q_r}_{1, \{H_r,H_i,H_j\}}-0+2I^{1,Q_r}_{1, \{H_r,H_i,H_{j-1}, H_1\}}\\
=2 \cdot 2 I^{0,Q_r}_{1, \{H_r,H_i,H_j\}}+2I^{0,Q_r}_{1, \{H_r,H_i,H_{j-1}\}}.
\end{multline}
The first summand is zero since $I^{0,Q_r}_{1, \{H_r,H_i,H_j\}}$ is $\frac{1}{2}$ the coefficient of $q^1H_0$ in $H_i\star H_j=4qH_1$.  The second summand is 2 since
$I^{0,Q_r}_{1, \{H_r,H_i,H_{j-1}\}}$ is $\frac{1}{2}$ the coefficient of $q^1H_0$ in $H_i\star H_{j-1}=H_r+2qH_0$.
\end{proof}

Note that this divisor satisfies Patterns I and II from \cite{Swin}.  Patterns I and II are intersection behavior shared by divisors lying on remaining extremal rays that he could not identify as being spanned by a conformal blocks divisor for $\sL_2$.
 In \cite{Swin}, the ray $R_{10}$ was identified as being spanned by a level 2 CB divisor for $\sL_6$.

\subsection{Examples for even quadrics}\label{EvenQuadrics}
Because of the cohomology class in the middle dimension, the classes for the even quadrics $X=Q_{2m}$ can be different, depending on whether $m$ is even or odd.  Moreover, when $m=2$, and $m=3$,  differences in the symmetry causes the classes to behave differently than in the general case.  To compute classes, the following facts are used.
\begin{lemma}
\begin{enumerate}
\item $H\star \xi_1 =H\star\xi_2 = \frac{1}{2}H_{m+1}$ (for degree reasons there are no $q$ terms).
\item $H_m\star H_m=H_{2m} +2q\cdot H_0$.
\item If $m$ is odd, then  $\langle \xi_1,\xi_2, [pt]\rangle_1=0$, and hence $\langle \xi_1,\xi_1, [pt]\rangle_1=1$. Therefore $\xi_1\star \xi_2=[pt]$ and $\xi_1\star\xi_1=\xi_2\star\xi_2=q\cdot 1$
\item If $m$ is even then
$\langle \xi_1,\xi_1, [pt]\rangle_1=0$, and hence $\langle \xi_1,\xi_2, [pt]\rangle_1=1$. Therefore
$\xi_1\star \xi_2=q\cdot 1$ and $\xi_1\star\xi_1=\xi_2\star\xi_2=[pt]$
\end{enumerate}
\end{lemma}
\begin{proof}
We need to compute $\xi_1\star\xi_2$ and $\xi_1\star\xi_1= \xi_2\star \xi_2.$ But $H_m\star H_m=(\xi_1+\xi_2)\star (\xi_1+\xi_2) = 2\xi_1\star\xi_2 + (\xi_1\star \xi_1 +\xi_2\star\xi_2)$, Therefore the $q\cdot 1$ terms in $\xi_1\star\xi_2$ and $\xi_1\star\xi_1$ add to $1$, so one of them should be one and the other $0$.
In the second case (the first is similar) pick linear spaces $M$ and $M'$ in the quadric $Q_r$ in general position and with cohomology class
$\xi_1$. We get linear spaces $M,M'\subseteq \Bbb{C}^{2m+2}$ of dimension $m+1$ each. The dimension of intersection of $M$ and $M'$ is congruent modulo two to $m+1$, an odd number. Therefore we may assume $M\cap M'$ is one dimensional. Now we want to count lines $L$ in the quadric through $M$, $M'$ and a general point $A$ in $Q_r$. Consider the span of $A$ and $M$ giving us a $P=\Bbb{P}^{m+1}$ in $\Bbb{P}^{2m+1}$. The quadric restricted to  $P$ equals $MT$, $T$ a hyperplane in $P$ which contains $A$. $M'\cap P$ is entirely contained in $M$, and we may assume that it does not intersect $T\cap M$. The line $L$ has to stay in $T$, and pass through $M'\cap P$ which does not intersect $T$. This is not possible.
\end{proof}

\subsubsection{A GW divisor on an extremal face of the nef cone spanned by conformal blocks divisors}

To compute the class of $I^{1,Q_6}_{2, \{H_1,H_6,\xi_1,\xi_1,\xi_1,\xi_1\}}$ in the nonadjacent basis, we intersect with the dual curves to see that

\begin{multline}
I^{1,Q_6}_{2, \{H_1,H_6,\xi_1,\xi_1,\xi_1,\xi_1\}}=[2, 0, 2, 2 , 0, 2, 4, 0, 4, 0, 0,  0, 6, 0, 0, 3]\\
=[1, 0 ,1 , 1 , 0 , 1 , 1 , 0 , 1 , 0 , 0 , 0 , 2 , 0,  0,  0] + [1  ,0 , 1 , 1 , 0 , 1 , 1 , 0 , 1 , 0 , 0 , 0 , 2 , 0 , 0 , 1] \\
+ [0 , 0 , 0 , 0 , 0 , 0 , 2 , 0 , 2 , 0 , 0 , 0 , 2 , 0 , 0 , 2 ]
=R_{1} +  R_{5} +2 R_{6}.
\end{multline}

When intersecting $D$ with the dual F- curves, we get:
\begin{multicols}{3}
\begin{itemize}
\item $F_{H_1,H_6,\xi_1,\{\xi_1,\xi_1,\xi_1\}} = 2$;
\item $F_{H_1,\xi_1 ,\{H_6,\xi_1\} ,\{\xi_1,\xi_1\}}= 0 $;
\item $F_{H_1, \xi_1,\xi_1 ,\{H_6,\xi_1,\xi_1\}} = 2$;
\item $F_{H_6, \xi_1, \xi_1, \{H_1,\xi_1,\xi_1\}} =2$;
\item $F_{H_6, \xi_1, \{H_1,6\}, \{\xi_1,\xi_1\}}= 0$;
\item $F_{H_1, H_6, \xi_1, \{\xi_1,\xi_1,\xi_1\}} = 2$;
\item $F_{\xi_1, \xi_1, \xi_1, \{H_1,H_6,\xi_1\}} =4$;
\item $F_{\xi_1, \xi_1, \{H_1,H_6\}, \{\xi_1,\xi_1\}} = 0$;
\item $F_{\xi_1, \xi_1, \xi_1, \{H_1,H_6,\xi_1\}}=4$;
\item $F_{\xi_1, \xi_1, \{H_1,H_6\}, \{\xi_1,\xi_1\}} = 0$;
\item $F_{\xi_1, \xi_1, \{H_1,H_6\}, \{\xi_1,\xi_1\}} = 0 $;
\item $F_{H_1, H_6, \{\xi_1,\xi_1\}, \{\xi_1,\xi_1\}} = 0 $;
\item $F_{H_6, \xi_1, \{H_1,\xi_1\}, \{\xi_1,\xi_1\}} = 0$;
\item $F_{H_1, \xi_1, \{H_6,\xi_1\}, \{\xi_1,\xi_1\}} = 0 $;
\item $F_{\xi_1, \xi_1, \{H_1,\xi_1\}, \{H_6,\xi_1\}} =3$;
\end{itemize}
\end{multicols}

For the coefficient of $\delta_{135}$, we have
\begin{multline}(F_{\xi_1, \xi_1, \{H_1,\xi_1\}, \{H_6,\xi_1\}} + F_{H_1, H_6, \xi_1, \{\xi_1,\xi_1,\xi_1\}}+F_{H_6, \xi_1, \xi_1, \{H_1,\xi_1,\xi_1\}} - F_{H_6, \xi_1,\{H_1,\xi_1\},\{ \xi_1,\xi_1\}})\\
=\frac{1}{2}I^{1,Q_6}_{1,\{H_3, H_4,\xi_1,\xi_1\}}+2+2-0=6.
\end{multline}

We note that $R_2$, $R_5$, and $R_6$ can all be expressed in terms of conformal blocks divisors.  Namely, for positive rational numbers $\rho_2$, $\rho_5$, and $\rho_6$,
\begin{itemize}
\item[$\op{R}_2$]  $=\rho_2 \ c_1(\mathbb{V}(\sL_{r+1},\{\omega_1^2, \omega_i, \omega_{r+1-i}, \omega_{r}^2\},2))$, with   $1 \le i \le \frac{(r+1)}{2}$;
\item[$\op{R}_5$]   $=\rho_5 \  c_1(\mathbb{V}(\sL_{r+1}, \{\omega_i^3,  \omega_{r+1-i}^3 \},1))$, with $r \ge 2$, and $i< \frac{r+1}{2}$;
\item[$\op{R}_6$]   $=\rho_6 \  c_1(\mathbb{V}(\sL_{r+1},\{ \ell \omega_1, m\omega_1, \ell \omega_r, m \omega_r, 0, 0\},\ell)$, with  $r\ge 1$.
\end{itemize}
In particular, this means that $I^{1,Q_6}_{2, \{H_1,H_6,\{\xi_1,\xi_1,\xi_1,\xi_1\}}$ is on an extremal face of the nef cone spanned by conformal blocks divisors.

\subsubsection{A GW class on an extremal face of the nef cone}\label{NEWfinally}
One can show that in the standard nonadjacent basis, for $X=Q_{4}$
\begin{multline}
I^{1,X}_{2, \{H_1,\xi_1,\xi_1,\xi_1,\xi_2,H_4\}}=
[0, 1, 1, 0,  2, 0,  2,  0,  2,  1,  2,   1,  2,   0,  0,   2]\\
 = [0 , 1,  0 , 0,  1,  0,  1,  0,  1,  1,  1,  1,  1,  0,  0,  1 ]+[0,  0,  1,  0,  1,  0,  1,  0,  1,  0,  1,  0 , 1 , 0,  0,  1]
=R_{20}+R_{3}.
\end{multline}
 The ray $R_{3}$ is known to be spanned by a conformal blocks divisor\footnote{$\op{R}_3 = \rho \ c_1(\mathbb{V}(\sL_{r+1},\{ \omega_1^3, \ell\omega_1,\omega_{r-2}, \ell \omega_r\},\ell))$, where $r\ge 3$, $\ell \ge 1$, for some positive rational $\rho$.}, but  $R_{20}$ is not known to be so.   In particular,
 we know of no description of any element on the interior of the extremal face spanned by $R_{3}$ and $R_{20}$ as being given by conformal blocks divisors.

\section{Two related questions}\label{Questions}
\subsection{Cycles from Gromov-Witten theory of Blow-ups}\label{QuestionsGWBlowUps}
Let $X$ be a convex variety (e.g., a homogenous projective variety) of dimension $m$, and $\pi:\widetilde{X}\to X$ the blow up of $X$ at a point $P\in X$. There is a natural inclusion
$\pi^*:A^*(X)\to A^*(\widetilde{X})$ via pull-back of cycles. Note that $\pi_*\circ \pi^*$ is  the identity on $A^*(X)$ (here $\pi_*:A^*(\widetilde{X})\to A^*(X)$ is the natural push forward map on cycles).

It follows from \cite[Lemma 2.2]{Gath} that if $\vec{\alpha}$ is an $n$-tuple of effective cycles in $A^*(X)$, and $\beta\in A_1(X)$, as we write in case the codimension $c$ cycle condition is satisfied by the triple $(X,\beta,\vec{\alpha})$,
\begin{multline}\label{gathmann}
I^{c,X}_{\beta,\{\alpha_1,\ldots,\alpha_n\}}
= I^X_{0,n,\beta}(\alpha_1\tensor\alpha_2\tensor\dots\tensor\alpha_n)\\
=I^{\widetilde{X}}_{0,n,\pi^*\beta}(\pi^*\alpha_1\tensor \pi^*\alpha_2\tensor\dots\tensor \pi^*\alpha_n)=I^{c,\widetilde{X}}_{\pi^*\beta,\{\pi^*\alpha_1,\ldots,\pi^*\alpha_n\}}.
\end{multline}

Now let $X=\Bbb{P}^r$. The cohomology of $X$ is generated by cycle classes of linear subspaces $L_d\subset \Bbb{P}^r$ of some codimension $d$. The cycle  classes of these linear spaces are the same,
 denoted $[L_d]$. Now choose one such linear subspace
$L_d\subset \Bbb{P}^r$ of codimension $d$ which passes through $P$, then $\pi^*[L_d]=L_d'+T_d$
where $L_d'\subset\widetilde{X}$ is the strict transform of  $L_d$, and  $T_d$ is  the class of a $\dim(L_d)=r-d$ linear subspace of the exceptional divisor, on $\widetilde{X}$ (one can identify the exceptional divisor with $\Bbb{P}^{r-1}$). Therefore if $\alpha_i$ are cycle classes of positive dimensional subspaces $L_{a_i}$ in $\Bbb{P}^r$ of co-dimension $a_i$ (so no point classes), then one can rewrite Equation \eqref{gathmann}  as follows,
\begin{multline}\label{qtty}
I^{c,X}_{\beta,\{\alpha_1,\ldots,\alpha_n\}} = I^X_{0,n,\beta}(\otimes_{i=1}^n\alpha_i)=I^{\widetilde{X}}_{0,n,\pi^*\beta}(\otimes_{i=1}^n\pi^*\alpha_i)\\
=I^{\widetilde{X}}_{0,n,\pi^*\beta}(\otimes_{i=1}^n (L'_{a_i}+T_{a_i})) = \sum_{S \subset \{1,\ldots,n\}} I^{\widetilde{X}}_{0,n,\pi^*\beta}\bigl(\left(\otimes_{i\in S} L_{a_i}' \right) \otimes \left( \otimes_{i\in S^c} T_{a_i}\right)\bigr).
\end{multline}
We have therefore decomposed the Gromov-Witten classes into a sum of (possibly non effective) cycles on $\ovop{M}_{0,n}$ by expanding the above quantity \eqref{qtty}.
\subsection{Fedorchuk's divisors}
For $\sum_i a_i =(r+1)(d+1)$, recall we have shown
\begin{equation}\label{Fdiv}
I^{1,\mathbb{P}^r}_{d, \vec{a}} \equiv c_1(\mathbb{V}(\operatorname{sl}_{r+1}, \{\omega_{a_1},\ldots, \omega_{a_n}\},1)).
\end{equation}
Assuming that none of the $a_i$ are zero, we consider the divisor
$$\Bbb{D}'=2c_1(\mathbb{V}(\operatorname{sl}_{r+1}, \{\omega_{a_1},\ldots, \omega_{a_n}\},1))-\sum_{r+1\mid \sum_{i\in I} a_i}\Delta_{I,J},$$
which Fedorchuk \cite[Equation (7.0.17)]{Fed} has proved is nef, and an effective sum of boundary classes.  However $\Bbb{D}'$ is not known to be semi-ample (i.e., that some multiple is basepoint free).  Using the scaling identity  \cite[Proposition 1.3]{GG} $$2c_1(\mathbb{V}(\operatorname{sl}_{r+1}, \{\omega_{a_1},\ldots, \omega_{a_n}\},1))=c_1(\mathbb{V}(\operatorname{sl}_{2r+2}, \{\omega_{2a_1},\ldots, \omega_{2a_n}\},1)),$$
we can rewrite the expression with Fedorchuk's divisor as (with $\vec{m}=2\vec{a}$),
$$I^{1,\mathbb{P}^{2r+1}}_{d, \vec{m}}= \Bbb{D}'+ \sum_{r+1\mid \sum_{i\in S} a_i}\Delta_{S,S^c}.$$
Can $\Bbb{D}'$  be characterised by Gromov-Witten theory of blow-ups, for example, is  $\Bbb{D}'$ equivalent to some combination of terms in the following natural decomposition (use \eqref{qtty} for $X=\Bbb{P}^{2r+1}$)?
\begin{equation}\label{Fed}
I^{1,\mathbb{P}^{2r+1}}_{d, 2\vec{m}}=\sum_{\overset{S \subsetneq \{1,\ldots,n\},}{} }  I^{\widetilde{X}}_{0,n,\pi^*\beta} \bigl(\left(\otimes_{i\in S} L_{m_i}' \right)\otimes \left(\otimes_{i\in S^c} T_{m_i}\right)\bigr).
\end{equation}
\subsection{Divisors from the Gromov-Witten theory of pairs}\label{QuestionsGWPairs}

Consider the case of a projective space $X=\Bbb{P}^r$ and a hyperplane $H$ in $X$. Let $s>1$ and $\alpha=(\alpha_1,\dots,\alpha_s)$ be an $s$-tuple of positive integers such that $\sum_{i=1}^s \alpha_i=d$. Define the  space $\overline{M}_{0,n,s}(X, d\mid \alpha)=\overline{M}_{0,n,s}(H/X, d\mid \alpha)$ to be the closure in $\overline{M}_{0,n+s}(X,d)$ of the set of {irreducible} stable maps $(C,x_1,\dots,x_n,y_1,\dots,y_s,f)$ of degree $d$ to $X$ with $f(C)\not\subset H$ such that the divisor $f^*H$ on $C\cong \Bbb{P}^1$ is equal to $\sum_i \alpha_i y_i$ (equality of cycles, not just linear equivalence). This implies $f(y_i)\in H$ (since $\alpha_i$ are assumed to be positive).

Vakil \cite{V} has shown that each irreducible component of
$\overline{M}_{0,n,s}(X, d\mid \alpha)$ has the expected dimension, which is equal to
$\dim \overline{M}_{0,n}(X,d)-\sum_{i=1}^s (\alpha_i-1).$ Let $\gamma_1,\dots,\gamma_n\in A^*(X)$ and $\mu_1,\dots,\mu_s\in A^*(H)$ and set
$\sum_i\operatorname{codim}\gamma_j+\sum_j \operatorname{codim}\mu_i=\tau$. Then one can form the cycle
$$(\operatorname{ev}^*_{x_1}\gamma_1\dots \operatorname{ev}^*_{x_n}\gamma_n)\cdot (\operatorname{ev}^*_{y_1}\mu_1\dots \operatorname{ev}^*_{x_s}\mu_s)\cap [\overline{M}_{0,n,s}(X,d\mid \alpha)]\in A^*(\overline{M}_{0,n,s}(X,d\mid \alpha)),$$
which has homological degree $\dim \overline{M}_{0,n,s}(X,d\mid \alpha)-\tau$,  pushforward to $\overline{M}_{0,n+s}$ the same degree,
and is a class of codimension $c$ if $\dim \overline{M}_{0,n,s}(X,d\mid \alpha)-\tau=\dim \overline{M}_{0,n+s}-c$, which simplifies to
$$d(r+1)+r+c= \sum\alpha_i +\sum \gamma_j +\sum \mu_i.$$
Let
$I^{c,H/X}_{d,\alpha}(\gamma_1\tensor\dots\tensor \gamma_n\mid\mu_1\tensor\dots\tensor \mu_s)\in A^c (\overline{M}_{0,n+s})$
denote the push-forward cycle. It is easy to see that it is effective (by Kleiman's theorem). However, it is not clear that $I^{c,H/X}_{d,\alpha}(\gamma_1\tensor\dots\tensor \gamma_n\mid\mu_1\tensor\dots\tensor \mu_s)$   is basepoint free.  To prove so using our methods so far, one would need to know the dimension of fibers of $\overline{M}_{0,n,s}(H/X, d\mid \alpha)\to\overline{M}_{0,n+s}$, or show this map is flat.

\subsubsection*{Acknowledgements} We thank Han-Bom Moon for his remarks on a draft of this paper. Gibney is supported by NSF.

\end{document}